\documentclass[reqno]{amsart} 

\usepackage{color}
\usepackage{xcolor}

\usepackage{attrib}

\usepackage{ifpdf}
\ifpdf 
\usepackage{graphicx,import} 
\usepackage[pdftex,     
plainpages=false,   
breaklinks=true,    
colorlinks=true,
linkcolor=blue,
citecolor=green,
pdftitle={A general Frobenius' Theorem via the Transport of Currents},
pdfauthor={Paolo Bonicatto}
]{hyperref} 
\else 
\usepackage{graphicx}   
\usepackage{hyperref}
\fi 

\usepackage{amsfonts,amsmath}	
\usepackage{amssymb}
\usepackage{verbatim}
\usepackage{amsopn}
\usepackage[english]{babel}
\usepackage{amsthm}
\usepackage{enumerate}
\usepackage{mathrsfs}	
\usepackage{mathtools}
\usepackage{harpoon}
\usepackage{esint}
\usepackage{todonotes}
\usepackage{cancel}

\usepackage{stmaryrd}
\usepackage{bm}
\usepackage{bbm}
\usepackage{caption}
\usepackage{subcaption}
\usepackage{nicefrac}
\usepackage[a4paper,left=35mm,right=35mm,top=34mm,bottom=40mm,marginpar=25mm]{geometry}

\newcommand{\Crm}{\mathrm{C}}

\newcommand{\Lrm}{\mathrm{L}}
\newcommand{\Mrm}{\mathrm{M}}
\newcommand{\Nrm}{\mathrm{N}}

\newcommand{\Wrm}{\mathrm{W}}

\newcommand{\Lcal}{\mathcal{L}}
\newcommand{\Mcal}{\mathcal{M}}

\newcommand{\Mbf}{\mathbf{M}}

\newcommand{\altnorm}[1]{{\left\vert\kern-0.25ex\left\vert\kern-0.25ex\left\vert #1 \right\vert\kern-0.25ex\right\vert\kern-0.25ex\right\vert}}
\newcommand{\eps}{\varepsilon}

\newcommand{\dpr}[1]{\langle #1 \rangle}

\newcommand{\dprb}[1]{\bigl\langle #1 \bigr\rangle}

\renewcommand{\rho}{\varrho}

\theoremstyle{definition} \newtheorem{definition}{Definition}[section]
\theoremstyle{definition} \newtheorem{remark}[definition]{Remark}
\theoremstyle{plain} \newtheorem{lemma}[definition]{Lemma}
\theoremstyle{plain} \newtheorem{proposition}[definition]{Proposition}
\theoremstyle{plain} \newtheorem{theorem}[definition]{Theorem}
\theoremstyle{plain} \newtheorem{corollary}[definition]{Corollary}
\theoremstyle{definition} 
\theoremstyle{plain} 
\theoremstyle{definition} 
\theoremstyle{plain} 
\theoremstyle{plain}

\DeclareMathOperator{\AC}{AC}
\DeclareMathOperator{\BV}{BV}

\DeclareMathOperator{\dive}{div}

\DeclareMathOperator{\Lip}{Lip}

\DeclareMathOperator{\curl}{curl}
\DeclareMathOperator{\Wedge}{{\textstyle\bigwedge}}

\DeclareMathOperator{\spn}{span}

\newcommand{\ee}{\mathrm{e}}

\newcommand{\sbullet}{\begin{picture}(1,1)(-0.5,-2.5)\circle*{2}\end{picture}}
\newcommand{\frarg}{\,\sbullet\,}

\newcommand{\R}{\mathbb{R}}

\newcommand{\N}{\mathbb{N}}

\renewcommand{\L}{\mathscr L}

\newcommand{\dd}{\mathrm{d}}

\newcommand{\T}{T}

\newcommand{\weaksto}{\overset{*}{\rightharpoonup}}

\newcommand{\Dscr}{\mathscr{D}}

\renewcommand{\L}{\mathscr L}

\theoremstyle{plain} \newtheorem*{theorem*}{Theorem}
\theoremstyle{plain} 
\theoremstyle{plain} \newtheorem*{mthm*}{Main Theorem}
\theoremstyle{plain} \newtheorem*{conjecture*}{Conjecture}
\theoremstyle{plain} 
\theoremstyle{plain} \newtheorem*{problem*}{Problem}

\newcommand{\bbb}[1]{\llbracket #1 \rrbracket}

\numberwithin{equation}{section}

\usepackage{color}
\usepackage{graphicx}
\usepackage{tikz}
\usepackage{framed}
\definecolor{shadecolor}{rgb}{0.94, 0.97, 1.0}

\usepackage{pict2e}
\makeatletter
\DeclareRobustCommand{\intprod}{%
\mathbin{\mathpalette\int@prod{(0.1,0)(0.9,0)(0.9,0.8)}}}
\DeclareRobustCommand{\restrict}{%
\mathbin{\mathpalette\int@prod{(0.1,0.8)(0.1,0)(0.9,0)}}}	
\newcommand{\int@prod}[2]{%
\begingroup
\sbox\z@{$\m@th#1+$}%
\setlength\unitlength{\wd\z@}%
\begin{picture}(1,1)
	\roundcap
	\polyline#2
\end{picture}%
\endgroup
}
\makeatother

\title[]{A general Frobenius' Theorem \\ via the transport of currents}

\author{Paolo Bonicatto}
\address[P.\ Bonicatto]{Università di Trento, Dipartimento di Matematica,
Via Sommarive 14, 38123 Trento, Italy}
\email{paolo.bonicatto@unitn.it}

\date{\today}

\begin{document}
\begin{abstract}
	A classical result in Differential Geometry states that the flows of two smooth vector fields commute if and only if their Lie Bracket vanishes. In this work, we extend this result to a more general setting where one of the vector fields is bounded and Lipschitz, while the other may be a singular vector-valued measure, i.e. a normal 1-current. This result is achieved via the study of two distinct evolutionary PDEs describing the transport of vector quantities (the Vector Advection Equation and the Geometric Transport Equation). Furthermore, we show that a celebrated Theorem by Alfvén in Magnetohydrodynamics can be interpreted as a suitable time-dependent version of Frobenius’ Theorem. Our approach builds on recent advances concerning the Geometric Transport Equation for currents \cite{BDNR2, BDNR}.	
	
\end{abstract}

\maketitle

\tableofcontents

\section{Introduction}
Given a vector field $b \colon \R^d \to \R^d$ (the reader could think of it as the velocity field of some fluid) and another vector field $\overline{v} \colon \R^d \to \R^d$ (the initial state of a certain vector quantity), we want to describe the \emph{transport} of $\overline{v}$ along $b$ under minimal regularity assumptions on $b$. From the mathematical point of view, this problem can be formulated in two ways:
\begin{enumerate}
	\item One may view $\overline{v} $ directly as a vector field; 
	\item Alternatively, one may think of $\overline{v}  = \overline{v} \L^d$ as a vector measure (1-current). 
\end{enumerate}
This distinction gives rise to the following two equations: the \emph{Vector Advection Equation}
\begin{equation}\label{eq:VAE_intro}\tag{VAE}
	\frac{\dd}{\dd t} v_t + [v_t,b] = 0, \qquad t \in (0,1)
\end{equation}
describes the transport of the vector field $\overline{v}$, while the \emph{Geometric Transport Equation} 
\begin{equation}\label{eq:GTE_intro}\tag{GTE}
	\frac{\dd}{\dd t} T_t + \mathcal{L}_b T_t = 0,  \qquad t \in (0,1)
\end{equation}
describes the transport of the 1-current $T_{\overline{v}}:=\overline{v} \L^d$. Here, in \eqref{eq:VAE_intro} the term  $[v_t,b]:= \nabla v_t \cdot b - \nabla b \cdot v_t$ denotes the (classical) Lie Bracket, whilst in \eqref{eq:GTE_intro} $\Lcal_b T_t$ denotes the Lie Derivative of a current $T_t$ in the direction $b$. This operator has been introduced in \cite{BDNR}, and can be defined by 
\[
\Lcal_b T_t := -\partial(b \wedge T_t) - b \wedge \partial T_t, \qquad \text{for every $t \in (0,1)$.}
\]
See below for more notation on currents. 
If all quantities were smooth, the solutions to the corresponding initial value problems starting from $\overline{v}$ would be given respectively by
\begin{equation}\label{eq:stellina_intro}
	v(t,x) = (\nabla X_t \cdot \overline{v})(X_{-t}(x)),
\end{equation}
in the case of \eqref{eq:VAE_intro}, and
$$
T_t = (X_t)_{\ast}(T_{\overline{v}}) 
= (\nabla X_t \cdot \overline{v})(X_{-t}(x)) \, \rho_t(x) \, \L^d, 
$$
for \eqref{eq:GTE_intro}. Here $X_t(x)=X(t,x)$ denotes the (classical) flow of $b$, and $\rho_t(x)=\rho(t,x)$ is the density of the Lebesgue measure along the flow, i.e. it holds $(X_t)_{\#}\L^d = \rho_t \L^d$ for every $t \in (0,1)$. More explicitly, one can show that 
\[
\rho(t,x) = \frac{1}{\det (\nabla X_t)(X_{-t}(x))}, \qquad (t,x) \in (0,1) \times \R^d.
\]

In the case of incompressible fluid flows, i.e. when $\dive b=0$, one easily verifies that the two equations coincide and indeed $\rho \equiv 1$. In the compressible case, however, the two equations do differ, the reason being that \eqref{eq:GTE_intro} encodes also the transport of the Lebesgue measure (which is compressed or dilated along the flow of $b$ by a factor governed precisely by $\dive b$). 

The first step out of the smooth framework is the study of \eqref{eq:VAE_intro} and \eqref{eq:GTE_intro} in the case of Lipschitz velocity fields, where the classical theory of ODEs ensures existence and uniqueness of the trajectories of $b$. The works \cite{BDN, BDNR2} contain well-posedness results for \eqref{eq:GTE_intro} within this setting. For \eqref{eq:VAE_intro}, one can start by giving a suitable notion of weak solution. In order to make sense of the Lie Bracket in a pointwise (at least a.e.) sense, it is natural to look for solutions $v_t \in\Wrm^{1,1}_x$. Our first result, obtained passing through the theory developed for \eqref{eq:GTE_intro}, yields that the representation formula \eqref{eq:stellina_intro} holds also in the case of Lipschitz velocity fields. 
\begin{proposition}[Representation formula for \eqref{eq:VAE_intro}] 
	If $b \in \Lip(\R^d;\R^d)$ is a bounded velocity field, and $v \in \Crm_t^0(\Wrm_x^{1,1})$ is a solution to \eqref{eq:VAE_intro} starting from $v_0 := \overline{v}$, then the representation formula \eqref{eq:stellina_intro} holds true for all $t \in (0,1)$ and $\L^d$-almost every $x \in\R^d$.
\end{proposition}
The fundamental idea is that, by letting the Lie derivative act on currents, one can assign a rigorous meaning to the Lie Bracket $[v,b]$, even when the full derivative of $v$ does not exist in a classical sense. Every integrable vector field (with measure divergence) naturally induces a normal 1-current and therefore the expression $\Lcal_b (T_v)$ is always well-defined. In this way, one is naturally led from the study of \eqref{eq:VAE_intro} to the one of \eqref{eq:GTE_intro}. In general, the guiding principle is that multiplying by the density $\rho$ allows one to pass from a solution to \eqref{eq:VAE_intro} to a solution to \eqref{eq:GTE_intro}. Combined with the well-posedness results available for \eqref{eq:GTE}, this constitutes a major stepping stone toward the desired conclusions. This passage however might introduce additional divergence (boundary) terms and, as a consequence, requires a delicate analysis, see Section \ref{s:vae} below.

\subsection*{Frobenius Theorem} 
The evolutionary framework we develop here around \eqref{eq:VAE_intro} yields a generalised version of Frobenius’ Theorem. Let us recall that, given two smooth vector fields $b,v$ on $\R^d$ with globally defined flows $X_t, Y_s$ respectively, one has 
\[
X_t \circ Y_s = Y_s \circ X_t \quad \text{for every $s,t \in \R$ as maps on $\R^d$} 
\iff [v,b] = 0.
\]
This is a classical result in Differential Geometry, often referred to as \emph{Frobenius' Theorem}, see, e.g., \cite[Proposition 18.5]{Lee}. In recent years, much attention has been devoted to possible generalisations of this statement beyond the smooth setting, replacing the classical notion of flow with the by-now standard notion of \emph{Regular Lagrangian Flow} (RLF). The celebrated DiPerna-Lions-Ambrosio theory ensures existence and uniqueness of RLFs for Sobolev or $\BV$ vector fields with bounded divergence, and it is therefore natural to ask whether an analogue of the Frobenius' Theorem can be formulated in these contexts.

\subsubsection*{State of the Art}
In \cite{RampazzoSussmann}, the authors established a Frobenius-type theorem when both vector fields are Lipschitz. They introduced a subtle set-valued notion of the bracket in order to cope with the lack of differentiability everywhere for Lipschitz functions. More recently, Colombo and Tione in \cite{ColomboTione} proved that if $b,v \in \Wrm^{1,p}(\R^d;\R^d)$ are Sobolev ($p \ge 1$) and bounded vector fields with bounded divergence, then the commutativity of the RLFs of $b,v$ holds if, and only if,  
\[
[b,v] = 0 \text{ and $X_t$ is weakly differentiable in direction $v$ with bounded derivative,}
\]
where $X_t$ is the RLF of $b$. 
See also \cite{RebucciZizza} for an interesting discussion related to the 2D setting, where additional structural rigidity allows one to remove the differentiability assumption of \cite{ColomboTione}. Particularly relevant to the present work is the contribution of \cite{RST}, where the authors rule out the possibility of obtaining a Frobenius-type theorem for vector fields admitting a unique a.e. flow (even when they are RLFs). In addition, they present an extension of the results of \cite{RampazzoSussmann} by addressing the mixed case in which one vector field is Lipschitz and the other is merely Sobolev.

\subsubsection*{Frobenius' Theorem via the Vector Advection Equation}
As a byproduct of our analysis of \eqref{eq:VAE_intro}, we obtain a version of Frobenius' Theorem in the case where one vector field is Lipschitz, but the transported object is a general normal 1-current. In this respect, our contribution is close in spirit to and extends \cite{RST}. In the case of vector fields, our commutativity result reads as follows.

\begin{theorem*}[Generalised Frobenius' Theorem]
	Let $b \in \Lip(\R^d;\R^d)$ be a bounded vector field and let $\rho$ denote its density. Let $v \colon \R^d \to \R^d$ be a bounded vector field, with bounded divergence, possessing a unique RLF and such that the currents $T_t := \rho_t v \L^d$ solve \eqref{eq:GTE_intro}. Then the RLFs of $b$ and $v$ commute.  
\end{theorem*}  

This statement becomes especially transparent in the case when $\dive b =0$ (i.e., $\rho \equiv 1$): every vector field $v$ with bounded divergence admitting a unique RLF and satisfying $\Lcal_b (T_v) = 0$ (if $v \in \Wrm^{1,1}$ this coincides precisely with the vanishing-bracket condition $[b,v]=0$ a.e.) is admissible, and hence its flow and the one of $b$ commute. Notably, this version of Frobenius' Theorem requires neither Sobolev nor BV assumptions on the transported field $v$, but only a control on its divergence (besides the uniqueness of its RLF).
\vspace{.1cm}

At this point, it is natural to push the current-based formulation even further. One can state a Frobenius-type result where the velocity field is Lipschitz and the transported objects are general 1-currents (i.e. possibly singular vector measures). The analogue in this context (once again, particularly clear in the incompressible case) can be stated as follows: if a normal 1-current $\overline{T}$ satisfies $\Lcal_b \overline{T}=0$, then the current is \emph{invariant} under the flow of $b$, i.e. $(X_t)_{*}\overline{T} = \overline{T}$. This invariance admits also a Lagrangian interpretation which can be seen by means of the Smirnov's Theorem \cite{smirnov}. According to this classical result, every normal 1-current can be written as a suitable superposition (without cancellations) of elementary currents associated to Lipschitz curves. We show that, given an incompressible Lipschitz field $b$ and a normal 1-current $\overline{T}$ such that $\Lcal_b \overline{T}=0$, if we push forward any Smirnov representation of $\overline{T}$ through the flow at time $t$ (we think here of the flow lifted to a map from curves to curves) we obtain another representation of the same current. In this sense, every Smirnov measure of $\overline{T}$ induces a trajectory of measures $\eta_t$, all of which are Smirnov measures of the same current, see Proposition \ref{prop:smirnov}.

\subsection*{Alfvén’s Theorem as a Time-Dependent Frobenius-Type Result} 
Our approach to Frobenius' Theorem via evolutionary-type PDEs like \eqref{eq:GTE_intro} can also serve to clarify certain results in Physics. A notable example is \emph{Alfvén’s Theorem} in Magnetohydrodynamics (MHD), which can be interpreted as a time-dependent analogue of Frobenius' Theorem.
In three dimensions, the Geometric Transport Equation \eqref{eq:GTE} can be used to describe the evolution of the magnetic field in a perfectly conducting fluid (like ideal plasma) advected by a velocity field. More precisely, the magnetic field $B = B(t,x)$ can be seen as a  divergence-free field (by Maxwell’s equations) satisfying the \emph{induction equation}
\begin{equation}\label{eq:induction_eq_intro}
	\frac{\dd}{\dd t} B_t = \curl(V\times B_t),
\end{equation}
where $V \colon \R^d \to \R^d$ represents the velocity of the fluid. As noted in \cite{BDNR}, equation \eqref{eq:induction_eq_intro} is precisely \eqref{eq:GTE_intro} in the case of absolutely continuous, boundaryless 1-currents $T_t := B_t \L^3$.
According to Alfvén's principle, in an ideal plasma with infinite conductivity, the matter of the fluid is \emph{fastened} to the magnetic field lines, or vice versa, the magnetic field lines are ‘frozen into the fluid.’
We provide a rigorous Lagrangian formulation of this statement (see Theorem~\ref{thm:alfven}) in the case of Lipschitz incompressible velocity fields $V$ and a possibly rough magnetic field $B$. Suppose that, at each time $t$, the magnetic field $B_t$ admits a unique (globally defined) RLF, denoted by $Y^{(t)}$. Then, for every $t \in (0,1)$, the uniqueness results for \eqref{eq:GTE_intro} and similar computations to those behind the generalised Frobenius' Theorem yield that 
\[
Y^{(t)} = X_t \circ Y^{(0)} \qquad \text {as maps on $\R \times \R^3$,}
\]
for every $t \in (0,1)$. In other words, a magnetic field line at time $0$ is carried by the flow of the fluid to a magnetic field line at time $t$, and this provides the mathematical interpretation of the statement that magnetic field lines are \emph{frozen} into the fluid.
In the compressible case, however, this no longer holds as stated: only the direction of the magnetic field lines is frozen, while their intensity is affected by the density factor $\rho$. Other Eulerian formulations of Alfvén’s Theorem will be investigated in future work.

\subsection*{Acknowledgements}
The author wishes to thank Andrea Marchese and Elio Marconi for helpful discussions on the topic of this paper and Martina Zizza and Giacomo Del Nin for their useful comments on a preliminary version of the manuscript.  

\section{Notation and preliminaries} \label{sc:prelim}
This section fixes our notation and recalls some basic facts. We refer the reader to~\cite{Federer69book,KrantzParks08book} for the complete proofs, as well as to \cite{BDNR2,BDNR} for further details on the Geometric Transport Equation.

Throughout the paper we will let $d \in \N$ be the ambient dimension. 
Given a vector $w \in \R^d$ we often write $(w)_i$ to denote the $i$-th component of $w$, where $i = 1,\ldots, d$.

\subsection{Currents} We refer to~\cite{Federer69book} for a comprehensive treatment of the theory of currents, summarising here only the main notions that we will need. The space of \textbf{$k$-dimensional currents} $\mathscr D_k(\R^d)$ is defined as the dual of $\Dscr^k(\R^d)$, where the latter space is the space of compactly supported smooth differential $k$-forms, endowed with the locally convex topology induced by local uniform convergence of all derivatives. Then, the notion of \textbf{(sequential weak*) convergence} is the following: 
\begin{equation*}
	T_n \weaksto T \text{ in the sense of currents } \Longleftrightarrow \dprb{ T_n, \omega} \to \dprb{ T, \omega} \quad  \text{for all $\omega \in \Dscr^k(\R^d)$.}
\end{equation*}
The \textbf{boundary} of a current is defined as the adjoint of De Rham's differential: if $T$ is a $k$-current, then $\partial T$ is the $(k-1)$-current given by 
\begin{equation*}
	\dprb{ \partial T, \omega } = \dprb{ T, d\omega },   \qquad  \omega \in \Dscr^{k-1}(\R^d). 
\end{equation*}
We denote by $\Mrm_k(\R^d)$ the space of $k$-\textbf{currents with finite mass} in $\R^d$, where the \textbf{mass} of a current $T \in \mathscr{D}_k(\R^d)$ is defined as 
\begin{equation*}
	\mathbf{M}(T):= \sup \left\{ \dprb{ T,\omega }:  \omega \in \Dscr^k(\R^d), \| \omega\|_\infty \le 1 \right\}.
\end{equation*}
Let $\mu$ be a finite measure on $\R^d$ and let ${\tau} \colon \R^d \to \Wedge_k(\R^d)$ be a map in ${\rm L}^1(\mu)$, where $\Wedge_k(\R^d)$ is the space of $k$-vector fields (for $k \in \{0, \ldots, d\}$). Then we define the current $T:={\tau}\mu$ as 
\begin{equation*}
	\dprb{ T, \omega } = \int_{\R^d} \dprb{ \tau(x),\omega(x) }\; \dd \mu(x), \qquad \forall \omega \in \Dscr^k(\R^d). 
\end{equation*}
We recall that all currents with finite mass can be represented as $T= \tau \mu$ for a suitable pair $\tau, \mu$ as above. In the case when $\|\tau\|=1$ $\mu$-a.e., we denote $\mu$ by $\|T\|$ and we call it the \textbf{mass measure} of $T$. As a consequence, we can write $T=\vec{T}\|T\|$, where $\|\vec{T}(x)\|=1$ for $\|T\|$-a.e. $x$. 

Given a current $T=\tau \mu \in \Dscr_k(\R^d)$ with finite mass and a vector field $v\colon \R^d \to \R^d$ defined $\|T\|$-a.e., we define the \textbf{wedge product}
\[
v\wedge T:=(v\wedge \tau) \mu\in \Dscr_{k+1}(\R^d).
\]

\subsubsection*{Absolutely continuous 1-currents}
In the following, given a vector field $v \colon \R^d \to \R^d$ we denote by $T_v:=v\L^d$ the canonical 1-current associated with $v$. We will adopt the same notation also when $v=v(t,x)$ is a time dependent vector field, i.e. a map $v \colon \R \times \R^d \to \R^d$.
Observe that, by a direct computation, for every smooth vector field $v$ we have 
\[
\partial T_v = -(\dive v)\L^d 
\]
in the sense of $0$-currents on $\R^d$. 

\subsubsection*{Normal currents}
A $k$-current on $\R^d$ is said to be \textbf{normal} if both $T$ and $\partial T$ have finite mass. The space of normal $k$-currents is denoted by $\mathrm N_k(\R^d)$.
The weak* topology on the space of (normal) currents has good properties of compactness and lower semicontinuity: if $(T_j)_j$ is a sequence of currents with $\Mbf(T_j) + \Mbf(\partial T_j) \le C < +\infty$ for every $j \in \N$, then there exists a normal current $T$ such that, up to a subsequence, $T_j \weaksto T$. Furthermore, 
\begin{equation*}
	\Mbf(T) \le\liminf_{j \to +\infty} \Mbf(T_j) ,  \qquad
	\Mbf(\partial T) \le\liminf_{j \to +\infty} \Mbf(\partial T_j).
\end{equation*}

\subsubsection*{Pushforward of currents}
The \textbf{pushforward} of $\T$ with respect to a proper $\Crm^1$-map $f\colon \R^d \to \R^d$ is defined by
\[
\dprb{ f_* T,\omega}=\dprb{ \T,f^* \omega}. 
\]
One can show the estimate 
\begin{equation}\label{eq:mass_of_pushforward}
	\Mbf(f_{*} T) \leq \int_{\R^d} \bigl\| Df(x)[\tau(x)] \bigr\| \; \dd \|T\|(x)
	\leq \|Df\|_\infty^k \Mbf(T).
\end{equation}
In the case of measures, we employ instead the standard notation $f_{\#} \mu$ for the pushforward of $\mu$ under a map $f$, namely, the measure defined by $f_{\#}\mu (A)=\mu(f^{-1}(A))$. Observe that $f_{*} \mu$ and $f_{\#} \mu$ coincide in this case.

\subsection{Decomposability bundle}
We recall from~\cite{AM} the definition and a few basic facts about the decomposability bundle. Given a measure $\mu$ on $\R^n$, the \textit{decomposability bundle} is a $\mu$-measurable map $x\mapsto V(\mu,x)$ (defined up to $\mu$-negligible sets) which associates to $\mu$-a.e.\ $x$ a subspace $V(\mu,x)$ of $\R^n$. We refer to~\cite{AM} for the precise notion of $\mu$-measurability of subspace-valued maps.

The map  $V$ satisfies the following property: Every Lipschitz function $f \colon \R^n\to\R$ is differentiable at $x$ along the subspace $V(\mu,x)$, for $\mu$-a.e.\ $x\in\R^n$. Moreover, this map is $\mu$-maximal in a suitable sense, meaning that $V(\mu,x)$ is, for $\mu$-a.e.\ $x$, the biggest subspace with this property (see~\cite[Theorem~1.1]{AM}). 
The directional derivative of $f$ at $x$ in direction $v \in V(\mu,x)$ will be denoted by $Df(x)[v]$. Observe that this is a slight abuse of notation, as the full differential $Df$ might not exist at $x$, even though the directional derivative exists.

A key fact about the decomposability bundle with regard to the theory of normal currents is the following~\cite[Theorem~5.10]{AM}: Given a normal $k$-current $T=\vec{T}\|T\|$ in $\R^n$, it holds that
\begin{equation}\label{eq:span_in_bundle}
	\spn(\vec{T})\subseteq V(\|T\|,x)\qquad\text{for $\|T\|$-a.e.\ $x\in\R^n$}.
\end{equation}
In particular, given any Lipschitz function $f$, we can define $D_T f$ at $\|T\|$-a.e.\ point as the restriction of the differential of $f$ to $\spn(\vec T)$. We will usually just write $Df$ instead of $D_T f$ when this differential is evaluated in a direction in $\spn(\vec T)$.

Recall that for a normal current $T\in \Nrm_k(\R^d)$ it is possible to define the pushforward $f_*T$ when $f\colon \R^d \to \R^d$ is merely Lipschitz via the homotopy formula~\cite[4.1.14]{Federer69book}. Classically, no explicit formula for this pushforward was available. However, it is shown in~\cite[Proposition~5.17]{AM} that the pushforward formula, in fact, remains true:

\begin{lemma}\label{lemma:pushforward_currents}
	Suppose that $T=\tau\mu$ is a normal $k$-current in $\R^n$, and $f \colon \R^n\to \R^m$ is a proper, injective Lipschitz map. Then, the pushforward current $f_*T$ satisfies
	\[
	f_*T= \tilde\tau\tilde\mu,
	\]
	where $\tilde\mu =f_\#\mu $, and $\tilde \tau(y)=Df(x) [\tau(x)] = D_T f(x) [\tau(x)]$ with $y=f(x)$.
\end{lemma}

\subsection{Flows of Lipschitz fields}\label{ss:flows}

Suppose that $b \colon \R^d \to \R^d$ is a globally bounded, Lipschitz vector field, and let 
$X \colon \R \times \R^d \to \R^d$ be the associated flow, i.e.~the unique map satisfying 
\begin{equation}
	\left\{
	\begin{aligned}
		\frac{\dd}{\dd t} X(t,x) &= b(X(t,x)), && (t,x) \in \R \times \R^d, \\ 
		X(0,x) &= x, && x \in \R^d. 
	\end{aligned}
	\right.
\end{equation}
The existence and uniqueness of the flow map (which is indeed defined on the whole $\R$) follow from the classical Cauchy--Lipschitz theory. As a direct consequence of uniqueness, one deduces the semigroup property
\[
X_t \circ X_s = X_{t+s}, \qquad t,s \in \R. 
\]
Moreover, for every $t \in \R$, the map $X_t$ is invertible, and it holds
\[
(X_t)^{-1} = X_{-t}.
\]
By Gronwall's inequality, for each fixed $t \in \R$, the map  
\[
X_t \colon x \mapsto X_t(x) := X(t,x)
\]
is Lipschitz. In particular, one can define the function $\rho \colon \R \times \R^d \to \R^+$ by
\[
\rho(t,x) := \frac{1}{\det\big(\nabla X_t\big)(X_{-t}(x))}. 
\]
By standard estimates and Rademacher’s Theorem, $\rho(t,\cdot)$ is well-defined and strictly positive for $\L^d$-a.e.~$x$. We remark that $\rho$ is not only strictly positive a.e. but also bounded from above and below by positive constants depending only on $t$ and the Lipschitz constant of $b$. Indeed, the Jacobian determinant of the flow can be controlled explicitly and one gets 
\[
e^{-\|\dive b\|_{\infty} t} \leq \rho(t,x) \leq e^{\|\dive b\|_\infty t}, 
\]
for every $t>0$ and $\L^d\text{-a.e. } x \in \R^d$. A short computation using the change of variables formula (or the Area Formula, see \cite[Section~3]{AmbrosioCrippa}) shows that $\rho$ is the unique solution (in the class of measure-valued solutions) to 
\[
\begin{cases} 
	\partial_t \rho + \dive (\rho b) = 0, \\ 
	\rho(0,\cdot) \equiv 1, 
\end{cases}
\]
that is, $\rho$ is precisely the density of the push-forward measure 
\[
(X_t)_{\#} \L^d = \rho(t,\cdot) \, \L^d. 
\]

\subsection{Regular Lagrangian Flows} 
Let $v \colon [0,1] \times \R^d \to \R^d$ be a bounded Borel vector field. A map 
\[
X \colon [0,1] \times \R^d \to \R^d
\]
is called a \emph{regular Lagrangian flow} (RLF) associated to $v$ if:
\begin{enumerate}
	\item For $\L^d$-a.e.~$x \in \R^d$, the map $t \mapsto X(t,x)$ is absolutely continuous and solves the ODE
	\[
	\frac{\dd}{\dd t} X(t,x) = v(t,X(t,x)), 
	\qquad X(0,x) = x.
	\]
	\item There exists a constant $C > 0$ (independent of $t$) such that 
	\[
	X(t,\cdot)_{\#} \L^d \leq C \, \L^d, \qquad \forall t \in [0,1].
	\]
\end{enumerate}
The DiPerna-Lions and Ambrosio's theory, see, e.g., \cite{AmbrosioCrippa}, ensures the existence and uniqueness of RLFs under mild assumptions such as $b \in \Lrm^1_t\Wrm^{1,1}_x$ or $\Lrm^1_t\BV_x$ with bounded divergence (and suitable growth conditions). 

\subsection{Geometric Transport Equation}

We recall that the Geometric Transport Equation has been introduced in \cite{BDNR} and reads as 
\begin{equation}\label{eq:GTE}\tag{GTE}
	\frac{\dd}{\dd t}T_t+\Lcal_b T_t =0,  \qquad t \in (0,1). 
\end{equation}
Given a Lipschitz and bounded vector field $b \colon \R^d \to \R^d$, we understand \eqref{eq:GTE} in the following sense: A time-parametrised family $(T_t)_{t \in (0,1)}$ of normal $k$-currents in $\R^d$ is called a weak solution to~\eqref{eq:GTE} if the integrability condition
\begin{equation}\label{eq:assumption_integrability_masses}
	\int_0^1 \Mbf(T_t)+\Mbf(\partial T_t) \;\dd t<\infty
\end{equation}
and
\begin{equation}\label{eq:PDE_weak_formulation}
	\int_0^1 \dprb{ T_t,\omega} \, \psi'(t) \;\dd t-\int_0^1\dprb{ \Lcal_b  T_t, \omega} \, \psi(t) \;\dd t=0
\end{equation}
hold for all $\psi\in \Crm^1_c((0,1))$ and all $\omega\in \Dscr^k(\R^d)$.
According to~\cite[Lemma 3.5(i)]{BDNR} we can assume that the map $t\mapsto T_t$ is weakly$^*$-continuous. Therefore, it makes sense to consider the initial-value problem 
\begin{equation}\label{eq:initial_value_problem}
	\left\{\begin{aligned}
		\frac{\dd}{\dd t}T_t+\Lcal_b T_t &=0,  \qquad t \in (0,1),\\
		T_0 &=\overline T,
	\end{aligned}\right.
\end{equation}
where the initial condition is understood in the sense that $\overline T$ is the weak$^*$ limit of $T_t$ as $t\to 0$.
Let us recall the main result from \cite{BDNR2}: 
\begin{theorem}\label{thm:main_JFA} 
	Let $b\colon \R^d \to \R^d$ be a globally bounded and Lipschitz vector field with flow $X_t = X(t,\frarg) \colon \R^d \to \R^d$ and let $\overline T \in \Nrm_k(\R^d)$ be a $k$-dimensional normal current on $\R^d$. Then, the initial-value problem \eqref{eq:initial_value_problem}
	admits a solution $(T_t)_{t \in (0,1)} \subset \Nrm_k(\R^d)$ of normal $k$-currents, which is unique in the class of normal $k$-currents. The solution is given by the pushforward of the initial current under the flow, namely, $T_t=(X_t)_*\overline T$.
\end{theorem}

Later in this work, we will need the following existence and uniqueness result for the non-homogenenous problem (for our purposes it is enough to state it for smooth velocity fields):

\begin{theorem}\label{thm:duhamel} 
	Let $b\colon \R^d \to \R^d$ be a globally bounded and smooth vector field with flow $X_t = X(t,\frarg) \colon \R^d \to \R^d$ and let $\overline T \in \Mrm_k(\R^d)$ be a $k$-dimensional current on $\R^d$. Let $t \mapsto R_t$ be a path of currents with 
	\[
	\int_0^1 \Mbf(R_t) \, \dd t < \infty. 
	\]
	Then, the initial-value problem 
	\begin{equation}\label{eq:inhomogenous_initial_value_problem}
		\left\{\begin{aligned}
			\frac{\dd}{\dd t}T_t+\Lcal_b T_t &=R_t,  \qquad t \in (0,1),\\
			T_0 &=\overline T,
		\end{aligned}\right.
	\end{equation}
	admits a solution $(T_t)_{t \in (0,1)} \subset \Mrm_k(\R^d)$, which is unique in the class of finite mass $k$-currents. The solution has the following Duhamel-type representation formula 
	\begin{equation}\label{eq:duhamel_formula} 
		T_t =(X_t)_*\overline T + \int_0^t (X_{t-s})_{*}(R_s) \, \dd s, \qquad t \in (0,1).
	\end{equation}
\end{theorem}

Recall that, in the case of smooth and autonomous vector fields, the Lie Derivative of a current can be equivalently defined by duality as 
\[
\dpr{\Lcal_bT, \omega} := -\dpr{T, \Lcal_b \omega}, \quad \text{ for each $\omega \in \Dscr^k(\R^d)$}.
\]
As a consequence, an element of $\Lrm^{1}(\Mrm_k(\R^d))$ is a weak solution to the equation in \eqref{eq:inhomogenous_initial_value_problem} if and only if for each $\omega \in \Dscr^k(\R^d)$ the map $t \mapsto \dpr{T_t, \omega}$ is absolutely continuous on $(0,1)$ and the following equality holds for $\L^1$-a.e. $t \in (0,1)$: 
\[
\frac{\dd }{\dd t} \dpr{T_t, \omega} = \dpr{T_t, \Lcal_b \omega} + \dpr{R_t,\omega}.
\]
See, e.g., \cite[Lemma 3.3]{BDNR}.
\begin{proof}[Proof of Theorem \ref{thm:duhamel}]
	Without loss of generality we can assume $\overline{T}=0$. 
	We begin by proving that \eqref{eq:duhamel_formula} defines indeed a solution to \eqref{eq:inhomogenous_initial_value_problem}. Fix a test $k$-form $\omega \in\Dscr^k(\R^d)$ and call $G: (0,1) \times (0,1) \to \R$ the map  
	\[
	G(s,t) := \dpr{(X_{t-s})_{*} R_s, \omega} = \dpr{R_s, (X_{t-s})^*\omega}. 
	\] 
	Clearly $G \in \Lrm^1((0,1)\times (0,1))$ and, for a.e. $s \in [0,1]$, the map $G(s,\frarg)$ is smooth on $(0,1)$. By some elementary computations -- see, e.g. the proof \cite[Theorem 3.6]{BDNR} -- it holds 
	\[
	\frac{\dd}{\dd t} (X_{t-s})^*\omega = (X_{t-s})^* (\Lcal_b \omega),
	\]
	which implies for $\L^1$-a.e. $s \in (0,1)$ and for all $t \in (0,1)$ 
	\[
	\frac{\partial G}{\partial t}(s,t) = \dpr{R_s, (X_{t-s})^* (\Lcal_b \omega)}. 
	\]
	By a straightforward application of the Lebesgue Dominated Convergence Theorem one can conclude that 
	the map 
	\[
	t \mapsto \int_0^t G(s,t) \, \dd s =\dpr{T_t,\omega} 
	\] 
	is absolutely continuous on $(0,1)$, with derivative given by 
	\[
	\frac{\dd}{\dd t} \dpr{T_t,\omega} = G(t,t) + \int_0^t \frac{\partial G}{\partial t}(s,t) \, \dd s \qquad \text{for $\L^1$-a.e. $t \in (0,1)$,}
	\]
	which means 
	\begin{align*}
		\frac{\dd}{\dd t} \dpr{T_t,\omega} & = \dpr{R_t, \omega} + \int_0^t  \dpr{R_s,  (X_{t-s})^* (\Lcal_b \omega)} \, \dd s \\ 
		& = \dpr{R_t, \omega} + \int_0^t  \dpr{(X_{t-s})_*R_s,  \Lcal_b \omega} \, \dd s  \\ 
		& = \dpr{R_t, \omega} +\dpr{T_t,  \Lcal_b \omega}
	\end{align*}
	as desired. 
	
	The uniqueness part is trivial: if $T^1_t$ and $T_t^2$ are two solutions to \eqref{eq:inhomogenous_initial_value_problem} (with $\overline{T}=0$) then their difference $S_t$ solves the homogeneous equation 
	\[
	\frac{\dd}{\dd t}T_t+\Lcal_b T_t = 0,  \qquad t \in (0,1)
	\]
	with trivial initial datum. By uniqueness for the homogeneous problem -- see, e.g., \cite[Theorem 3.6]{BDNR} -- the conclusion follows. \end{proof}

\subsection{A lemma on absolutely continuous functions}
This technical lemma will be needed in the following. The proof is standard and therefore omitted.
\begin{lemma}\label{lemma:AC_curves_composition} Let $f \colon \R^d \to \R$ be a Lipschitz function and let $\gamma \in \AC([0,1]; \R^d)$ be an absolutely continuous curve. Then the composition $f \circ \gamma \colon [0,1] \to \R$ is absolutely continuous and, in particular, it holds 
	\[
	f(\gamma(b))-f(\gamma(a)) = \int_a^b \frac{\dd}{\dd s} f(\gamma(s)) \dd s, \qquad \text{ for every } a,b \text{ with} \ 0 \le a \le b \le 1. 
	\]
\end{lemma}

\section{The Vector Advection Equation and the Geometric Transport Equation}\label{s:vae}

We begin this section by collecting a useful formula that connects the Lie derivative of a current with the Lie bracket in the smooth setting. We then introduce the Vector Advection Equation \eqref{eq:vector_advection}, together with the definition of weak solutions. This will allow us to establish a bridge with the Geometric Transport Equation and to derive several corollaries: in the smooth case we obtain an existence and uniqueness result for \eqref{eq:vector_advection}, while in the case of Lipschitz velocity fields we prove a uniqueness result together with a representation formula for the solution. 

\subsection{On the Lie derivative and the Lie Bracket of smooth vector fields}

Within this section we let $b,v \colon \R^d \to \R^d$ be smooth vector fields. The \textbf{Lie Bracket} between $b$ and $v$ is defined by 
\[
[b,v](x) := \nabla b(x) \cdot v(x)  - \nabla v (x)\cdot b (x), \qquad x \in \R^d. 
\]
Recall that, given a vector field $v \colon \R^d \to \R^d$ we denote by $T_v:=v\L^d$ the canonical 1-current associated with $v$. In particular, if $v$ is smooth we have $
\partial T_v = -(\dive v)\L^d$ in the sense of $0$-currents on $\R^d$. 

For future reference, we collect in the following proposition an explicit expression for the Lie derivative of $T_v$. 
\begin{lemma}\label{lemma:lie_bracket}
	Let $v,b \colon \R^d \to \R^d$ be vector fields of class $\Crm^\infty(\R^d;\R^d)$. Then 
	\[
	\Lcal_{b} (T_v) = (-[b,v] +(\dive b) v) \L^d
	\]
	in the sense of $1$-currents on $\R^d$.
\end{lemma}

\begin{proof} A direct computation shows that 
	\begin{align*}
		\Lcal_{b} (T_v) & = - \partial (b \wedge T_v) - b \wedge \partial T_v \\ 
		& = - \partial (b \wedge T_v) + (b \dive  v) \L^d\\ 
		& = - \partial ((b \wedge v) \L^d) + (b \dive  v) \L^d \\ 
		& = - (\dive (b \otimes v - v \otimes b) + b \dive v) \L^d \\
		& =  ((\dive b) v - (\dive v) b - [b,v] + b (\dive v))\L^d\\
		& = ((\dive b) v - [b,v])\L^d.
	\end{align*}
	Here we have used the following two identities:
	\begin{equation}\label{eq:identity_1}
		\partial ((b \wedge v) \L^d) = \dive (b \otimes v - v \otimes b)\L^d
	\end{equation}
	and 
	\begin{equation}\label{eq:identity_2}
		\dive (b \otimes v - v \otimes b) =  [b,v] + b(\dive v) - (\dive b)v. 
	\end{equation}
	Let us prove these two identities. Equality \eqref{eq:identity_1} follows testing both expressions on a test form $\omega$. By linearity it is enough to choose $\omega=\omega_i dx^i$ (for a fixed $i \in \{1,\ldots, d\}$). We have 
	\begin{align*}
		\langle \partial ((b \wedge v) \L^d), \omega \rangle  & = \int_{\R^d} \langle b \wedge v, d\omega \rangle \, \dd x \\ 
		& = \sum_j  \int_{\R^d}\frac{\partial \omega_i}{\partial x_j} \langle b \wedge v,  dx^j\wedge dx^i \rangle \, \dd x \\ 
		& = \sum_j  \int_{\R^d}\frac{\partial \omega_i}{\partial x_j} (b^j v^i - b^i v^j) \dd x\\ 
		& = - \sum_j \int_{\R^d} \frac{\partial \omega_i }{\partial x_j}  (b^iv^j - b^jv^i) \, \dd x \\ 
		& = - \int_{\R^d} \nabla \omega_i \cdot (b \otimes v - v \otimes b)_i \, \dd x \\
		& = \langle \dive (b \otimes v - v \otimes b), \omega \rangle. 
	\end{align*}
	Identity \eqref{eq:identity_2} follows instead from the Leibniz rule, as 
	\begin{align*}
		\dive (b \otimes v - v \otimes b)_{i} - [b,v]_{i} & = \sum_j  \partial_j (b^iv^j - b^jv^i) - (\partial_j b^i v^j - \partial_j v^i b^j) \\ 
		& = \sum_j b^i \partial_j v^j - \partial_j b^j v^i \\ 
		& = b^i(\dive v) - v^i \dive b, 
	\end{align*}
	for every $i=1,\ldots, d$. This clearly gives \eqref{eq:identity_2}.
\end{proof}

\begin{remark} It is worth remarking that the Lie Derivative $\Lcal_b(T_v)$ is \emph{not} anti-symmetric w.r.t. to $b$ and $v$, contrary to what happens with the standard Lie Derivative operator in Differential Geometry ($\Lcal_b v = [b,v] = - [v,b] = - \Lcal_v b$). In particular, it is not true in general that $\Lcal_b T_b = 0$, unless $\dive b= 0$. This is due to the fact that the Lie Derivative of the current $T_v$ keeps track also of the evolution of the mass measure of $T_v$ along the flow of $b$. The Lebesgue measure is not preserved under the flow of a Lipschitz field, but changes indeed according to a quantity measured by the divergence of the field -- and this explains the presence of the additional summand in the formula.
\end{remark}

\subsection{Vector Advection Equation}

We now turn our attention to the Vector Advection Equation. Throughout this section we consider a fixed velocity field $b \colon \R^d \to \R^d$, which is assumed to be (at least) Lipschitz and globally bounded. Its flow map at time $t$ and position $x$ is denoted by $X(t,x)=X_t(x)$.

If $v \in \Wrm^{1,1}(\R^d)$, the definition of Lie Bracket between $v$ and $b$ still makes sense at $\L^d$-almost every $x \in \R^d$ and the vector field $[v,b]$ belongs to $\Lrm^1(\R^d;\R^d)$. This allows one to consider the initial value problem for the {\bfseries vector-advection equation}, namely
\begin{equation}\label{eq:vector_advection}\tag{VAE}
	\left\{\begin{aligned}
		\frac{\dd}{\dd t} v_t + [v_t,b] &=0,  \qquad t \in (0,1),\\
		v_0 &=\overline{v},
	\end{aligned}\right.
\end{equation}
where $\overline{v}$ is a given initial datum. 

\begin{definition}[Weak solutions to \eqref{eq:vector_advection}] We say that a time-dependent family of vector fields $v \in \Crm^0([0,1];\Wrm^{1,1}(\R^d))$ is a \emph{weak solution} to
	\[
	\frac{\dd}{\dd t} v_t + [v_t,b] = 0 
	\]  
	if, for every $\Psi \in \Crm^\infty_c((0,1)\times \R^d;\R^d)$, it holds 
	\begin{equation}\label{eq:weak_VAE}
		\int_0^1 \int_{\R^d} v(t,x) \cdot \partial_t \Psi(t,x)  - [v_t,b](x) \cdot \Psi(t,x) \, \dd t \, \dd x = 0.
	\end{equation}
	In this case, we also say that $v$ is {\bfseries advected} by $b$.
\end{definition}

We remark that the continuity in time of the family $v_t$ is not essential in the definition of weak solution, as one can make sense of \eqref{eq:weak_VAE} also for $v \in \Lrm^1_t (W^{1,1}_x)$.
However, later in the paper we will often consider constant-in-time weak solutions and thus the definition above suffices for our purposes. The continuity in time of the family allows also one to naturally retrieve the initial condition in a pointwise sense: If the family $v$ is advected by $b$ and $v_0 = \overline{v}$ as elements of $\Wrm^{1,1}(\R^d)$, we say that $v$ is a weak solution to the initial value problem \eqref{eq:vector_advection}.

\subsection{From the Vector Advection Equation to the Geometric Transport Equation}

We now show that \eqref{eq:vector_advection} has a close connection with \eqref{eq:GTE}. Recall that the density $\rho \colon \R\times \R^d \to \R^+$ of the flow $X_t$ is defined by 
\[
\rho(t,x) := \frac{1}{\det(\nabla X_t)(X_{-t}(x))}. 
\]
The following propositions play a pivotal role because they allow to pass from a solution to \eqref{eq:vector_advection} to a solution to \eqref{eq:GTE}. We begin by studying the situation in the case of smooth velocity fields. 

\begin{lemma}\label{lemma:VAE_to_GTE}  Let $b \in \Crm^{\infty}(\R^d;\R^d)$ and $v \in \Crm^0([0,1];\Wrm^{1,1}(\R^d))$ be a family advected by $b$. Define the vector field 
	\[
	w(t,x) :=\rho(t,x)v(t,x)
	\]
	and let $T_{w_t}$ be the 1-current canonically associated to $w_t$. Then $t\mapsto T_{w_{t}}$ is a path of normal 1-currents transported by $b$, i.e. it holds 
	\begin{equation}\label{eq:GTE_for_w}
		\begin{aligned}
			\frac{\dd}{\dd t} T_{w_t} + \Lcal_{b}(T_{w_t}) & = 0,  \qquad t \in (0,1).
		\end{aligned}
	\end{equation} 
\end{lemma}

\begin{proof}
	Since $b$ is smooth, so is the density $\rho$. This yields immediately that, for every $t \in [0,1]$, we have 
	\[
	\partial T_{w_t} = - \dive(\rho_t v_t) = - \nabla \rho_t \cdot v_t - \rho_t \dive v_t \in \Lrm^1 (\R^d)  
	\]
	and thus the currents $T_{w_t}$ are normal. 
	Let us now show that they solve the equation \eqref{eq:GTE_for_w}. 
	
	First recall that the (smooth) function $\rho$ solves the continuity equation driven by $b$, therefore it holds 
	\begin{equation}\label{eq:continuity_equation}
		\partial_t \rho(t,x) = -\dive(\rho b)(t,x)
	\end{equation}
	pointwise everywhere. 
	
	We denote by $\alpha =\sum_i \alpha_i dx^i$ a smooth, compactly supported 1-form and by $a := \sum_i \alpha_i \ee^i \in \Crm^\infty_c(\R^d;\R^d)$ the associated 1-vector field. Let $\psi \in \Crm^\infty_c((0,1))$ be a test function. Observe that, since $\rho$ is smooth the function $\Psi(t,x) := \rho(t,x) \psi(t) a(x)$ is an admissible test function in \eqref{eq:weak_VAE} and we thus have 
	\begin{align*}
		0 & = \int_0^1 \int_{\R^d} v(t,x) \cdot \partial_t \Psi(t,x)  - [v_t,b](x) \cdot \Psi(t,x) \, \dd t \, \dd x \\ 
		& = \int_0^1 \int_{\R^d} v(t,x) \cdot a(x) \rho(t,x) \psi'(t)  +  v(t,x) \cdot a(x)\psi(t) \partial_t \rho(t,x) \\ 
		& \qquad -  \rho(t,x) \psi(t) [v_t,b](x) \cdot a(x) \, \dd t \, \dd x \\ 
		& = \int_0^1 \int_{\R^d} v(t,x) \cdot a(x) \rho(t,x) \psi'(t)  -  v(t,x) \cdot a(x)\psi(t)\dive(\rho b)(t,x) \\ 
		& \qquad - \rho(t,x) \psi(t) [v_t,b](x) \cdot a(x) \, \dd t \, \dd x \\ 
	\end{align*}
	where in the last line we have taken into account \eqref{eq:continuity_equation}. Integrating by parts the second summand and rearranging gives 
	\begin{align}
		\begin{aligned}\label{eq:utile}
			& \int_0^1 \int_{\R^d}  v(t,x) \cdot a(x) \rho(t,x) \psi'(t)  \, \dd t \, \dd x \\ 
			= & \int_0^1 \int_{\R^d}  \left[ \rho(t,x)  \psi(t) [v_t,b](x) \cdot a(x) -  \psi(t) \rho(t,x) b(x) \cdot \nabla (v(t,x) \cdot a(x)) \right] \, \dd t \, \dd x.	
		\end{aligned}
	\end{align}
	We are now ready to write the single summands of \eqref{eq:GTE_for_w}: on the one hand, directly from \eqref{eq:utile} we have 
	\begin{equation}\label{eq:time}
		\begin{split}
			\int_0^1 \dprb{ T_{w(t)},\alpha} \, \psi'(t) \;\dd t  & =  \int_0^1 \int_{\R^d} w(t,x) \cdot a(x) \, \psi'(t) \; \dd x \, \dd t  \\
			& =  \int_0^1 \int_{\R^d} \rho(t,x) v(t,x) \cdot a(x) \, \psi'(t) \;  \dd x \,  \dd t \\ 
			& =  \int_0^1 \int_{\R^d} [v,b](t,x) \cdot \rho(t,x)a(x) \psi(t)\, \dd x\, \dd t\\ 
			& \qquad -  \int_0^1 \int_{\R^d} \rho(t,x) b(x)\cdot \nabla (v(t,x)\cdot a(x)) \psi(t) \, \dd x\, \dd t .
		\end{split}
	\end{equation}
	On the other hand, for the term involving the Lie derivative, we first observe that a straightforward approximation argument yields that the conclusion of Lemma \ref{lemma:lie_bracket} still holds if $v \in \Wrm^{1,1}(\R^d)$. Therefore, we can use the formula relating the Lie derivative of $T_{w_t}$ (with respect to $b$) with the Lie bracket between $b$ and $w_t$. More precisely,
	\begin{equation}\label{eq:Lie}
		\begin{split}
			- & \int_0^1 \dprb{ \Lcal_b T_{w(t)}, \alpha} \, \psi(t) \;\dd t \\ 
			= & \int_0^1 \int_{\R^d} ([b(x), w(t,x)] - w(t,x) \dive b(x) )\cdot a(x)  \, \psi(t) \; \dd x \, \dd t \\ 
			= & \int_0^1 \int_{\R^d} [b(x), \rho(t,x) v(t,x)]\cdot a(x) \psi(t)\; \dd x \, \dd t \\ &
			\quad - \int_0^1 \int_{\R^d} \rho(t,x) \dive b(x) v(t,x) \cdot a(x)\psi(t) \; \dd x \, \dd t. 
		\end{split} 
	\end{equation}
	Let us compute separately the first integrand: since $\rho_t$ is smooth, $\rho_t v_t \in \Wrm^{1,1}(\R^d)$ for each $t$ and a direct application of the Leibniz rule yields 
	\[
	(\nabla(\rho v)(b))_i = \sum_j \partial_j (\rho v_i) b^j = \sum_j (\partial_j \rho) v_i b^j + \sum_j \rho (\partial_j v_i) b^j  = (\nabla \rho \cdot b) v_i + \rho (\nabla v (b))_i
	\]
	for every $i=1,\ldots, d$, and therefore 
	\begin{equation}\label{eq:integrand}
		\begin{split}
			[b, \rho v] = \nabla b(\rho v) - \nabla (\rho v) (b) 
			= \rho \nabla b(v) - \rho\nabla v (b) - (\nabla\rho \cdot b) v 
			= \rho [b,v] - (\nabla \rho \cdot b) v.
		\end{split}
	\end{equation}
	Plugging \eqref{eq:integrand} in \eqref{eq:Lie} and integrating by parts the third term, we obtain
	\begin{equation}\label{eq:Lie_final}
		\begin{split}
			- \int_0^1 \dprb{ \Lcal_b T_{w(t)}, \alpha} \, \psi(t) \;\dd t 
			& =  \int_0^1 \int_{\R^d} \rho(t,x)[b,v](t,x) \cdot a(x) \psi(t) \; \dd x \, \dd t \\
			& \quad  - \int_0^1 \int_{\R^d} (\nabla \rho(t,x) \cdot b(x)) v(t,x) \cdot a(x) \psi(t) \; \dd x \, \dd t  \\ 
			& \quad -  \int_0^1 \int_{\R^d} \rho(t,x) \dive b(x) v(t,x) \cdot a(x)\psi(t) \; \dd x \, \dd t \\ 
			= & \int_0^1 \int_{\R^d} \rho(t,x)[b,v](t,x) \cdot a(x) \psi(t) \; \dd x \, \dd t  \\ 
			& \quad \cancel{- \int_0^1 \int_{\R^d} (\nabla \rho(t,x) \cdot b(x)) v(t,x) \cdot a(x) \psi(t) \; \dd x \, \dd t } \\ 
			&\quad + \cancel{\int_0^1 \int_{\R^d}b(x) \cdot \nabla\rho(t,x) \, v(t,x) \cdot a(x) \psi(t) \; \dd x \, \dd t } \\ 
			& \quad +\int_0^1 \int_{\R^d} \rho(t,x) b(x) \cdot \nabla(v(t,x) \cdot a(x))  \psi(t) \; \dd x \, \dd t
			\\ 
			= & \int_0^1 \int_{\R^d} \rho(t,x)[b,v](t,x) \cdot a(x) \psi(t) \; \dd x \, \dd t  \\ 
			& \quad + \int_0^1 \int_{\R^d} \rho(t,x) b(x) \cdot \nabla(v(t,x) \cdot a(x)) \psi(t) \; \dd x \, \dd t. 
		\end{split} 
	\end{equation}
	Combining \eqref{eq:time} with \eqref{eq:Lie_final}, the terms involving the Lie Bracket cancel out because $[b,v]=-[v,b]$ and we obtain 
	\begin{align*}
		\int_0^1 \dprb{ T_{w(t)},\alpha} \, \psi'(t)  - \dprb{ \Lcal_b T_{w(t)}, \alpha} \, \psi(t) \;\dd t = 0. 
	\end{align*}
	By the usual density argument of forms of the kind $\alpha \psi$ (see, e.g., \cite[Lemma 3.3]{BDNR}) we have \eqref{eq:GTE_for_w} and this concludes the proof.
\end{proof}

As an immediate corollary, we obtain existence and uniqueness to \eqref{eq:vector_advection} in the case of smooth velocity fields.

\begin{corollary}\label{cor:esistenza_unicita_VAE} Let $b \in \Crm^\infty(\R^d;\R^d)$ be a smooth vector field and let $\overline{v} \in \Crm^\infty(\R^d;\R^d)$ be a smooth initial datum. Then there exists a unique solution to \eqref{eq:vector_advection}, and it is given by 
	\[
	v(t,x) = (\nabla X_{t} \cdot \overline{v})(X_{-t}(x)), 
	\]
	for every $(t,x) \in \R \times \R^d$. 
\end{corollary}

\begin{proof} In order to show existence, one can compute that $v(t,x) = (\nabla X_{t} \cdot \overline{v})(X_{-t}(x))$ does indeed satisfy the equation via a straightforward (albeit tedious) computation -- see e.g. \cite{chorin2013mathematical}. 
	Alternatively, one can argue as follows. Notice that the function $w(t,x):=v(t,x) \rho(t,x)$ (with $v$ defined as in the statement) induces a family of normal (smooth) currents which can be written as $(X_{t})_* (T_{\overline{v}})$ by the very definition of push-forward. Building on the theory for the Geometric Transport Equation one thus has that 
	\begin{equation*}
		\begin{aligned}
			\frac{\dd}{\dd t} T_{w_t} + \Lcal_{b}(T_{w_t}) & = 0,  \qquad t \in (0,1)
		\end{aligned}
	\end{equation*} 
	which can be written in coordinates as 
	\begin{equation*}
		\begin{aligned}
			\frac{\dd}{\dd t}{(\rho_t v_t)} - [b,\rho_t v_t] + (\dive b ) \rho_t v_t  & = 0,  \qquad t \in (0,1).
		\end{aligned}
	\end{equation*} 
	A simple application of the chain rule, together with the fact that the density $\rho$ satisfies the continuity equation, yields that 
	\begin{align*}
		\frac{\dd}{\dd t}{ v_t } & = 	\frac{\dd}{\dd t}\left( \frac{1}{\rho_t}(\rho_t v_t) \right) \\ 
		& = -\frac{\partial_t \rho_t}{\rho_t^2} \rho_t v_t + \frac{1}{\rho_t} \left([b,\rho_t v_t] - (\dive b)  \rho_t v_t \right) \\
		& = - v_t \frac{\partial_t \rho_t}{\rho_t} + \frac{1}{\rho_t} \left(\rho_t [b, v_t] - \nabla \rho_t \cdot b \, v_t  - (\dive b)  \rho_t v_t \right) \\
		& = [b,v_t] - v_t \frac{\partial_t \rho_t + \dive(\rho_t b)}{\rho_t} \\
		& = [b,v_t], 
	\end{align*}
	from which the conclusion follows. (Here we have taken advantage of the smoothness assumptions to work with pointwise everywhere equalities, which clearly imply the validity of the equation in the definition of weak solution). 
	
	For the uniqueness, one can invoke again the theory of currents: if $v=v(t,x)$ is a non-trivial solution starting from initial datum $0$, then $T_t:=\rho_t v_t$ would be a non-trivial solution starting from the zero current to GTE, in view of Lemma \ref{lemma:VAE_to_GTE}. This contradicts \cite[Theorem 3.6]{BDNR}, and therefore the claim is proven. \end{proof}

An approximation argument allows one to infer the following conclusion in the case of Lipschitz velocity fields:

\begin{proposition}[Uniqueness and representation formula for \eqref{eq:vector_advection}]\label{prop:uniqueness_to_VAE_in_Lip}
	Let $b \colon \R^d \to \R^d$ be Lipschitz and bounded. Let $v\in \Crm^0([0,1]; \Wrm^{1,1}(\R^d))$ be a weak solution to \eqref{eq:vector_advection} driven by $b$ with initial datum $\overline{v} \in \Wrm^{1,1}(\R^d)$. Then the family of currents $t \mapsto T_t:=T_{\rho_t v_t}$ is normal and solves \eqref{eq:GTE}. In particular, it holds  
	\[
	v(t,x) = (\nabla X_{t} \cdot \overline{v})(X_{-t}(x)), \qquad \text{ for every $t \in \R$ and $\L^d$-a.e. $x \in \R^d$.}
	\]
\end{proposition}

\begin{proof}		
	Consider a family of smooth mollifiers (in space) $\{\sigma^\eps\}_{\eps>0} \subseteq \Crm^{\infty}_c(\R^d)$ and set $b^\eps := b \ast \sigma^\eps$. Let $\rho^\eps$ be the smooth, uniformly positive density of $b^\eps$.
	The vector fields $w^\eps_t:=\rho_t^\eps v_t$ are still $\Crm^0_t\Wrm_x^{1,1}$ and the currents $T_{w_t^\eps}$ are normal for every $t \in [0,1]$. 
	We now claim that the currents $T_{w_t^\eps}$ solve 
	\begin{equation}\label{eq:nonhomogenous}
		\left\{\begin{aligned}
			\frac{\dd}{\dd t} T_{w^\eps_t} + \Lcal_{b^\eps}(T_{w^\eps_t}) & = R_t^\eps,  \qquad t \in \R,\\
			T_{w_0^\eps} &=T_{\overline{v}}, 
		\end{aligned}\right.
	\end{equation}
	where the reminder $R^\eps_t$ is the 1-current defined by $R^\eps_t:=( \rho_t^\eps[v_t,b^\eps-b])\L^d$ for every $\eps >0$.
	In order to show \eqref{eq:nonhomogenous}, it is enough to repeat similar computations to those in the proof of Lemma \ref{lemma:VAE_to_GTE}. For the sake of completeness we outline here the full argument. 
	We denote again by $\alpha =\sum_i \alpha_i dx^i$ a smooth, compactly supported 1-form and by $a := \sum_i \alpha_i \ee^i \in \Crm^\infty_c(\R^d;\R^d)$ the associated 1-vector field and we pick a test function $\psi \in \Crm^\infty_c((0,1))$.  
	On the one hand, for the term involving the time derivative nothing changes, and one has 
	\begin{equation*}
		\begin{split}
			\int_0^1 \dprb{ T_{w^\eps_t},\alpha} \, \psi'(t) \;\dd t  & =  \int_0^1 \int_{\R^d} w^\eps(t,x) \cdot a(x) \, \psi'(t) \; \dd x \, \dd t  \\
			& =  \int_0^1 \int_{\R^d} \rho^\eps(t,x) v(t,x) \cdot a(x) \, \psi'(t) \;  \dd x \,  \dd t \\ 
			& =  \int_0^1 \int_{\R^d} [v,b](t,x) \cdot \rho^\eps(t,x)a(x) \psi(t)\, \dd x\, \dd t\\ 
			& \qquad -  \int_0^1 \int_{\R^d} \rho^\eps(t,x) b^\eps(x)\cdot \nabla (v(t,x)\cdot a(x)) \psi(t) \, \dd x\, \dd t , 
		\end{split}
	\end{equation*}
	where we are using that $\rho^\eps$ is the solution to the continuity equation driven by $b^\eps$ and that $v$ is advected by $b$. 
	For the term involving the Lie derivative, we have instead 
	\begin{equation*}
		\begin{split}
			- \int_0^1 \dprb{ \Lcal_{b^\eps} T_{w^\eps_t}, \alpha} \, \psi(t) \;\dd t 
			= & \int_0^1 \int_{\R^d} ([b^\eps(x), w^\eps(t,x)] - w^\eps(t,x) \dive b^\eps(x) )\cdot a(x)  \, \psi(t) \; \dd x \, \dd t \\ 
			= & \int_0^1 \int_{\R^d} [b^\eps (x), \rho^\eps(t,x) v(t,x)]\cdot a(x) \psi(t)\; \dd x \, \dd t \\ 
			& \quad - \int_0^1 \int_{\R^d} \rho^\eps (t,x) \dive b^\eps(x) v(t,x) \cdot a(x)\psi(t) \; \dd x \, \dd t. 
		\end{split} 
	\end{equation*}
	Using the (pointwise a.e.) identity $[b^\eps, \rho_t^\eps v_t]= \rho_t^\eps [b^\eps,v_t] - (\nabla \rho_t^\eps \cdot b^\eps) v_t$, we obtain 
	\begin{equation*}
		\begin{split}
			- \int_0^1 \dprb{ \Lcal_{b^\eps} T_{w^\eps_t}, \alpha} \, \psi(t) \;\dd t
			& =  \int_0^1 \int_{\R^d} \rho^\eps(t,x) [b^\eps,v](t,x) \cdot a(x) \psi(t)\; \dd x \, \dd t \\ 
			& \quad -  \int_0^1 \int_{\R^d} (\nabla \rho^\eps(t,x) \cdot b^\eps(x)) v(t,x) \cdot a(x) \psi(t)\; \dd x \, \dd t  \\ 
			& \quad  - \int_0^1 \int_{\R^d} \rho^\eps (t,x) \dive b^\eps(x) v(t,x) \cdot a(x)\psi(t) \; \dd x \, \dd t. \\ 
			&= \int_0^1 \int_{\R^d} \rho^\eps(t,x) [b^\eps,v](t,x) \cdot a(x) \psi(t)\; \dd x \, \dd t \\ 
			& \quad - \cancel{ \int_0^1 \int_{\R^d} (\nabla \rho^\eps(t,x) \cdot b^\eps(x)) v(t,x) \cdot a(x) \psi(t)\; \dd x \, \dd t}  \\ 
			& \quad +\cancel{\int_0^1 \int_{\R^d} (\nabla \rho^\eps(t,x) \cdot b^\eps(x)) v(t,x) \cdot a(x) \psi(t)\; \dd x \, \dd t}  \\ 
			&\quad  + \int_0^1 \int_{\R^d} \rho^\eps(t,x) b^\eps(x) \nabla (v_t(x) \cdot a(x)) \psi(t)\; \dd x \, \dd t. 
		\end{split} 
	\end{equation*}
	All in all,
	\begin{align*}
		\int_0^1 \dprb{ T_{w^\eps_t},\alpha} \, \psi'(t) - \dprb{ \Lcal_{b^\eps} T_{w^\eps_t}, \alpha} \, \psi(t) \;\dd t & =  \int_0^1 \int_{\R^d} \rho^\eps(t,x) [v,b](t,x) \cdot a(x) \psi(t)\, \dd x\, \dd t\\ 
		& \quad - \cancel{\int_0^1 \int_{\R^d} \rho^\eps(t,x) b^\eps(x)\cdot \nabla (v(t,x)\cdot a(x)) \psi(t) \, \dd x\, \dd t} \\ 
		&  + \int_0^1 \int_{\R^d} \rho^\eps(t,x) [b^\eps,v](t,x) \cdot a(x) \psi(t)\; \dd x \, \dd t \\ 
		& \quad  + \cancel{\int_0^1 \int_{\R^d} \rho^\eps(t,x) b^\eps(x)\cdot \nabla (v(t,x)\cdot a(x)) \psi(t) \, \dd x\, \dd t} \\ 
		& =  \int_0^1 \int_{\R^d} \rho^\eps(t,x) a(x) \cdot [v,b-b^\eps](t,x) \psi(t)\; \dd x \, \dd t \\
		& = - \int_0^1 \dpr{R^\eps_t, \alpha} \psi(t) \, \dd t.  
	\end{align*}
	and this gives the desired conclusion. 
	Observe that 
	\[
	\int_0^1 \Mbf(R^\eps_t) \, \dd t = \int_0^1 \int_{\R^d} \rho^\eps(t,x) \big \vert [v,b-b^\eps](t,x) \big\vert \, \dd x \,\dd t  < \infty 
	\]
	and actually for every fixed $t$   
	\begin{align*}
		\Mbf(R^\eps_t) & = \int_{\R^d} \rho^\eps(t,x) \big \vert [v,b-b^\eps](t,x) \big\vert \, \dd x \\ 
		&  \le C(\Lip b) \left[  \| v_t \|_{\Wrm^{1,1}} \|b-b^\eps\|_{\infty}  + \int_{\R^d} |v(t,x)| |\nabla b(x)-\nabla b^\eps(x)| \, \dd x \right] 
	\end{align*}
	and both summands converge to $0$ as $\eps \to 0$ (the former because $b^\eps \to b$ strongly, the latter by Lebesgue Dominated Convergence Theorem, since $\nabla b^\eps(x) \to \nabla b(x)$ at a.e. $x$ by Rademacher's Theorem).
	
	Invoking the Duhamel-type formula of Theorem \ref{thm:duhamel} we obtain that for every $\eps>0$ it holds 
	\[
	T_{w^\eps_t } = (X^\eps_t)_{*} (T_{\overline{v}}) + \int_0^t (X^\eps_{t-s})_{*}(R_s^\eps)\, \dd s. 
	\]
	Observe that the following estimate holds: 
	\[
	\Mbf\left(  \int_0^t (X^\eps_{t-s})_{*}(R_s^\eps)\, \dd s \right)\le \int_0^t \Lip(X^\eps_{t-s}) \Mbf(R^\eps_s) \, \dd s \to 0 
	\]
	as $\eps \to 0$. For the other summand, notice that, since $X^\eps_t \to X_t$ uniformly (on compact sets), we have $(X^\eps_t)_{*} (T_{\overline{v}}) \weaksto (X_t)_{*} (T_{\overline{v}})$. In conclusion, 
	\[
	T_{w^\eps_t} \weaksto  (X_t)_{*} (T_{\overline{v}}) 
	\] 
	as $\eps \to 0$ weakly-star in the sense of currents. On the other hand, since $b$ is Lipschitz, we have $\rho_t^\eps \weaksto \rho_t$ for each $t$ and therefore for every $\Phi \in \Crm_c^\infty(\R^d;\R^d)$ (and fixed $t$) we have 
	\[
	\int_{\R^d} \rho^\eps(t,x) v(t,x) \cdot \Phi(x) \, \dd x  \to \int_{\R^d} \rho(t,x) v(t,x) \cdot \Phi(x) \, \dd x 
	\]
	because $v_t \cdot \Phi \in \Lrm^1(\R^d)$ for every $t$, which means 
	\[
	T_{w_t^\eps} \weaksto T_{w_t}, \qquad w_{t} := \rho_t v_t 
	\]
	weakly-star in the sense of currents for every $t \in [0,1]$. 
	By uniqueness of the limit, we conclude that 
	\[
	T_{w_t} =  (X_t)_{*} (T_{\overline{v}})
	\]
	for every $t$, i.e.
	\[
	\rho(t,\frarg) v(t,\frarg) \L^d = (\nabla X_{t} \cdot \overline{v})(X_{-t}(\frarg)) \rho(t,\frarg) \L^d. 
	\]
	Since $\rho \in [c,C]$ for suitable constants $c,C>0$, we have that for every $t \in [0,1]$, it holds  
	\begin{equation*}
		v_t(\frarg) = (\nabla X_t \cdot \overline{v})(X_{-t}(\frarg))  \qquad \text{$\L^d$-almost everywhere on $\R^d$}, 
	\end{equation*}
	which concludes the proof. 
\end{proof}

\section{Frobenius' Theorem}
We are now able to prove the following central result: 
\begin{theorem}[Generalised Frobenius' Theorem] \label{thm:big_one} Let $b \colon \R^d \to \R^d$ be Lipschitz and bounded. Let $v  \colon \R^d \to \R^d$ be a vector field such that: 
	\begin{enumerate}
		\item $v \in \Lrm^\infty(\R^d;\R^d)$ with bounded divergence $\dive v \in \Lrm^\infty(\R^d)$; 
		\item $v$ has exactly one Regular Lagrangian Flow, $Y=Y(s,y)$ defined on $[0,1] \times \R^d$; 
		\item the vector field $w(t,x) := \rho(t,x)v(x)$ induces a path of normal currents $t \mapsto T_t := T_{w_t}$ which are transported by $b$.  
	\end{enumerate}
	Then the flows of $b$ and $v$ commute, i.e. 
	\[
	X_t \circ Y_s = Y_s \circ X_t \qquad \text{for every $t,s \in [0,1]$, \, a.e. on $\R^d$}.
	\]
\end{theorem}

\begin{proof}
	We split the proof in several steps. 
	\vspace{.5em}
	
	{\bfseries \emph{Step 1. Invariance via \eqref{eq:GTE}}}. Assumption (3) means that the (continuous-in-time) family $t\mapsto T_{w_t}$ solves the initial value problem (recall that $\rho_0 \equiv 1$) 
	\[
	\left\{\begin{aligned}
		\frac{\dd}{\dd t}T_t+\Lcal_b T_t &=0,  \qquad t \in (0,1),\\
		T_0 &=T_v.
	\end{aligned}\right.
	\] 
	By Theorem \ref{thm:main_JFA}, we deduce that 
	\[
	T_{w_t} = (X_{t})_{*} T_v, 
	\] 
	for every $t$. In view of Lemma \ref{lemma:pushforward_currents} this readily implies 
	\[
	(\nabla X_t \cdot v)(X_{-t}(\frarg)) \rho(t,\frarg) \L^d = \rho(t,\frarg) v(\frarg) \L^d. 
	\]
	We have already observed that $\rho \in [c,C]$ for suitable constants $c,C>0$, therefore, for every $t \in [0,1]$, it holds  
	\begin{equation}\label{eq:stellina}
		(\nabla X_t \cdot v)(X_{-t}(\frarg))  =  v(\frarg) \qquad \text{ a.e. on $\R^d$.} 
	\end{equation}
	\vspace{.5em}
	
	{\bfseries \emph{Step 2. Images of integral curves.}} We now show that \eqref{eq:stellina} implies that, for every $t \in [0,1]$ and for a.e. $y \in \R^d$, the function 
	\[
	s \mapsto X_{t}(Y_s(y))
	\]
	is differentiable for $\L^1$-a.e. $s$ and it holds 
	\begin{equation}\label{eq:derivata_di_composizione_cruciale}
		\frac{\dd}{\dd s} X_{t}(Y_s(y)) = v(X_{t}(Y_s(y))). 
	\end{equation}
	Indeed, we have by definition 
	\begin{align*}
		\frac{\dd}{\dd s} (X_t(Y_s(y)) & = \lim_h \frac{X_t\left(Y_{s+h}(y))-X_t(Y_s(y)\right)}{h} \\ 
		& = \lim_h \frac{X_t\left(Y_{s}(y) + \int_s^{s+h} v(Y_r(y)) \dd r\right)- X_t(Y_s(y)+hv(Y_s(y)))}{h}  \\
		& \qquad + \lim_h \frac{ X_t(Y_s(y)+hv(Y_s(y)))- X_t(Y_s(y))}{h} \\ 
		& =: \ell + (DX_t \cdot v)(Y_s(y)), 
	\end{align*}
	provided that $\ell$ exists. Let us show that this is the case and that indeed $\ell=0$. We have 
	\begin{align*}
		& \frac{1}{h}\left|X_t\left(Y_{s}(y) + \int_s^{s+h} v(Y_r(y)) \dd r \right) - X_t(Y_s(y)+hv(Y_s(y)))\right| \\ 
		\le & \Lip(X_t) \left| \fint_s^{s+h} v(Y_r(y))\, \dd  r - v(Y_s(y))\right| 
	\end{align*}
	and the right-hand side goes to $0$ as $h \to 0$ at every Lebesgue point of $r \mapsto v(Y_r(y))$, and thus for a.e. $s$. 
	We have therefore obtained 
	\[
	\frac{\dd}{\dd s} (X_t(Y_s(y)) = (DX_t \cdot v)(Y_s(y)), 
	\]
	hence, invoking \eqref{eq:stellina}, the desired \eqref{eq:derivata_di_composizione_cruciale} follows
	for a.e. $y \in \R^d$ and for a.e. $s \in [0,1]$. Observe that we have here used the following standard Fubini-like argument: Let 
	$N_t:=\{x \in \R^d : \text{\eqref{eq:stellina} is not satisfied}\}$.   
	By assumption 
	\[
	\L^d(N_t) = 0 \implies ((Y_{s})_{\#}\L^d)(N_t) = 0 \implies  \L^d (Y_{-s}(N_t))=0 
	\]
	for every $s$, i.e. $\L^d (\{z : z = Y_{-s}(y), \text{ for some $y \in N_t$}\})=0$ i.e. $\L^d (\{z : Y(s,z)\in N_t \})=0$
	and thus 
	\[
	\int_0^1 \L^d (\{z: Y(s,z) \in N_t\})\, \dd s =0 .
	\]
	Fubini's Theorem now gives 
	\[
	\int_{\R^d} \L^1 (\{s \in [0,1]: Y(s,z) \in N_t\}) \dd z =0 
	\]
	whence  
	\[
	\L^1 (\{s \in [0,1]: Y(s,z) \in N_t\}) = 0
	\]
	for a.e. $z$. This ensures that the equality \eqref{eq:derivata_di_composizione_cruciale} is satisfied for $\L^1$-a.e. $s$.
	\vspace{.5em}
	
	{\bfseries \emph{Step 3. Regular Lagrangian Flow of $v$.}} Fix $t \in [0,1]$. For $s \in [0,1]$ consider the following maps defined on $\R^d$: 
	\[
	Z_s := X_{-t} \circ Y_s \circ X_t
	\]
	We claim that $Z=Z(s,y)$ is a Regular Lagrangian Flow of $v$. 
	\vspace{.5em}
	
	{\bfseries \emph{Step 3.1. Non-concentration of the image measure.}} Recall that $(X_{t})_{\#} \L^d = \rho_t \L^d$ and that $(Y_{s})_{\#} \L^d \le C \L^d$ for some $C>0$ for every $s$. Fix a continuous, non-negative function $\phi \in \Crm^0(\R^d)$. We have 
	\begin{align*}
		\int_{\R^d} \phi \, \dd ((Z_{s})_{\#} \L^d) & = \int_{\R^d} \phi(X_{t} \circ Y_{-s} \circ X_{-t})\,  \dd \L^d \\ 
		& = \int_{\R^d} \phi(X_{t} \circ Y_{-s})\, \rho_t \dd \L^d \\ 
		& \le C  \int_{\R^d} \phi(X_{t})\, \rho_t \dd \L^d \\ 
		& =  C  \int_{\R^d} \phi\,  \dd \L^d 
	\end{align*}
	which gives $(Z_{s})_{\#} \L^d \le C \L^d$, as we wanted.
	\vspace{.5em}
	
	{\bfseries \emph{ Step 3.2. Integral curves of $v$.}} Let us fix $a,b \in [0,1]$ with $a\le b$ and $x \in \R^d$ and set $y:=X_t(x)$.  Then 
	\[
	Z(b,x) - Z(a,x) = X_{-t}(Y_b(X_t(x))) - X_{-t}(Y_a(X_t(x))) = X_{-t}(Y_b(y)) - X_{-t}(Y_a(y)). 
	\]
	Since $s \mapsto Y_{s}(y)$ is absolutely continuous, by Lemma \ref{lemma:AC_curves_composition}, the curve $s \mapsto X_{-t}(Y_s(y))$ is still absolutely continuous, therefore
	\begin{align*}
		X_{-t}(Y_b(y)) - X_{-t}(Y_a(y)) & = \int_a^b \frac{\dd}{\dd s} X_{-t}(Y_s(y)) \, \dd s \\
		&  = \int_a^b v(X_{-t}(Y_{s}(y))) \, \dd s \\
		&  =  \int_a^b v(X_{-t}(Y_{s}(X_{t}(x)))) \, \dd s \\ 
		& = \int_a^b v(Z_s(x)) \, \dd s, 
	\end{align*}
	finally yielding
	\[
	Z(b,x) - Z(a,x) = \int_a^b v(Z_s(x)) \, \dd s.
	\]
	The conclusion now follows because, by uniqueness of the Regular Lagrangian Flow, we must have 
	\[
	Z(s,y) = Y(s,y) \qquad \text{ for every $s$ and $\L^{d}$-a.e $y$},  
	\]
	whence the commutativity. 
\end{proof}

\begin{corollary}[Classical Frobenius' Theorem]\label{cor:classical_frobenius} Let $b \colon \R^d \to \R^d$ be Lipschitz and bounded. Let $v \in \Wrm^{1,1}(\R^d;\R^d)$ have bounded divergence. If $[b,v]=0$ a.e. in $\R^d$, then the flows of $b$ and $v$ commute, i.e. 
	\[
	X_t \circ Y_s = Y_s \circ X_t \qquad \text{for every $t,s \in [0,1]$, \, a.e. on $\R^d$}.
	\]
\end{corollary}

\begin{proof}
	The assumption $[b,v]=0$ a.e. on $\R^d$ implies that the constant-in-time family $v_t \equiv v$ solves \eqref{eq:vector_advection}. By Proposition \ref{prop:uniqueness_to_VAE_in_Lip} the family of normal currents $t \mapsto T_{w_t}$, with $w(t,x):=\rho(t,x)v(x)$, is thus transported by $b$. The conclusion now follows directly from Theorem \ref{thm:big_one}. 
\end{proof}

\begin{remark}
	In the incompressible case, i.e. when $\rho \equiv 1$, it is easily seen that the generalised Frobenius' Theorem above encompasses a large class of vector fields. Indeed, every bounded vector field (with bounded divergence) possessing a unique Regular Lagrangian Flow is admissible. In this case, the current $T_{w_t} = T_{v}$ is normal and constant in time and if $\Lcal_b T_v = 0$ then the flows of $b$ and $v$ commute. Moreover, the proof of Corollary \ref{cor:classical_frobenius} can also be specialised to the case in which both $b$ and $v$ are Lipschitz continuous, thus obtaining an independent proof of Frobenius' Theorem within the Cauchy-Lipschitz theory -- and without the introduction of set valued Lie brackets as in \cite{RampazzoSussmann}.
\end{remark}

\begin{remark} Another consequence of Frobenius' Theorem concerns the evolution of the boundary of the current $T_v$. From the equality $\rho_t v \L^d = (X_{t})_{*}(T_v)$ we immediately deduce, taking the boundary, that 
	\[ 
	\dive (\rho_t v ) = (X_{t})_{\#}(\dive v \L^d)
	\]
	as 0-currents in $\R^d$. In particular, in the incompressible case, the divergence of $v$ is \emph{invariant} under the flow.
\end{remark} 

\subsection{Extension to normal 1-currents}
For Lipschitz velocity fields, we can exploit the well-posedness of \eqref{eq:GTE} in the class of general normal 1-currents (see Theorem \ref{thm:main_JFA}) to derive an even more general Frobenius-type theorem. 
We begin with the following simple statement. 
\begin{proposition}\label{prop:invariance}
	Let $b \in \Lip(\R^d;\R^d)$ be a bounded vector field with $\dive b =0$ and let $X=X(t,x)$ be its flow. Let $\overline{T} \in \Nrm_1(\R^d)$ be such that $\Lcal_b \overline{T} = 0$. Then it holds $\overline{T} = (X_{t})_{*} \overline{T}$ for every $t \in (0,1)$.
\end{proposition}
This is an immediate consequence of the uniqueness statement and representation formula of Theorem \ref{thm:main_JFA}.

We can use the following classical decomposition result to give a Lagrangian interpretation of the invariance relation $\overline{T} = (X_{t})_{*} \overline{T}$. Here, we denote $\Gamma:=\Lip ([0,1];\R^d)$ and we define $\bbb{\gamma}$ as the natural 1-current associated with a Lipschitz curve $\gamma \colon [0,1] \to \R^d$, i.e., the functional defined by 
\[
\dpr{\bbb{\gamma}, \omega} := \int_\gamma \omega = \int_0^1 \omega (\gamma (s)) \cdot \gamma'(s) \, \dd s, \qquad \forall \omega \in \Dscr^1(\R^d).
\] 
(Equivalently, $\bbb{\gamma}$ is the integral 1-current $[\gamma([0,1]), \tau, 1]$ where $\tau$ is a unit tangent vector to the image of $\gamma$.) 

\begin{proposition}[Decomposition of normal 1-currents in curves] \label{prop:thm_di_smirnov} Let $\overline{T} \in \Nrm_1(\R^d)$ be a normal 1-current. Then there is a finite positive measure $\eta$ on $\Gamma$ with
	\[
	\overline{T} = \int_{\Gamma} \bbb{\gamma} \, \dd \eta(\gamma)
	\]
as $1$-currents on $\R^d$ and 	
	\[
	\|\overline{T}\| = \int_{\Gamma} \|\bbb{\gamma}\| \, \dd\eta(\gamma)
	\]
as measures on $\R^d$. Moreover, the following properties hold: 
\begin{enumerate}
	\item For $\eta$-a.e. $\gamma$, it holds  
	\[
	\| \bbb{\gamma} \| = \mathscr H^{1}\restrict \gamma([0,1]) = \gamma_{\#}(|\gamma'|\L^1 \restrict [0,1]). 
	\]
	In particular the total mass $\|\bbb{\gamma}\|(\R^d)$ equals the \emph{length} of the curve $\gamma$; 
	\item If we write $\overline{T}=\tau \, \mu$ for a unit vector field $\tau$ defined $\mu$-almost everywhere, then for $\eta$-almost every $\gamma \in \Gamma$ it holds 
	\[
	\gamma'(s) = |\gamma'(s)|\tau(\gamma(s)) \, \qquad \text{for a.e. $s \in [0,1]$}.
	\]
\end{enumerate}
\end{proposition}
See, for instance, \cite[Theorem C]{smirnov} and \cite[Lemma 9.1]{BG} for a direct proof of the last assertion. In what follows, we will refer to $\eta$ as a \emph{measure representation} of the current $\overline{T}$.
\begin{definition} Let $b \in \Lip(\R^d;\R^d)$ be a bounded vector field and let $X=X(t,x)$ be its flow. We call \emph{lifted flow of $b$ at time $t\in [0,1]$} the map $\mathtt{X}_t \colon \Gamma \to \Gamma$ defined by
	\[
	\mathtt{X}_t(\gamma)(s):= X_t(\gamma(s)).
	\]
\end{definition}
Notice that the definition is well-posed, namely for every $\gamma \in \Gamma$ the curve $\mathtt X_t(\gamma) \in \Gamma$, since for every $s,r \in [0,1]$ we have 
\[
|\mathtt X_t(\gamma)(s)-\mathtt X_t(\gamma)(r) | = |X_t(\gamma(s))-X_t(\gamma(r))| \le \Lip (X_t) |\gamma(s)-\gamma(r) | \le \Lip(X_t) \Lip(\gamma)|s-r|. 
\]
The map $\mathtt X_t$ is also seen to be Borel, since it is Lipschitz continuous w.r.t. the sup-norm on $\Gamma$, as 
\[
\|\mathtt X_t(\gamma_1)-\mathtt X_t(\gamma_2) \|_\infty \le \Lip (X_t)  \|\gamma_1 - \gamma_2 \|_{\infty}.
\]
In particular, given a (finite) measure $\eta \in \Mcal(\Gamma)$ we can consider its pushforward 
\[
(\mathtt X_t)_{\#}\eta =: \eta_t  
\]
which is still a (finite) measure on $\Gamma$.

We can now prove the following proposition, which shows that $\eta_t$ is always a measure representation of $T_t$. 

\begin{proposition}[Frobenius' Theorem for normal 1-currents]\label{prop:smirnov} Let $b \in \Lip(\R^d;\R^d)$ be a bounded vector field and let $X=X(t,x)$ be its flow. Let $\overline{T} \in \Nrm_1(\R^d)$ be a normal 1-current and let $\eta \in \Mcal (\Gamma)$ be a measure representation of $\overline{T}$. 
	Let $T_t$ be the unique solution to 
	\[
	\left\{\begin{aligned}
		\frac{\dd}{\dd t}T_t+\Lcal_{b} T_t &=0,  \qquad t \in (0,1),\\
		T_0 &=\overline{T}.
	\end{aligned}\right.
	\]
	Then $\eta_t$ is a measure representation of $T_t$ for a.e. $t \in [0,1]$. 
\end{proposition}

\begin{proof}
	Let us consider $\overline{T}$ as in the statement and let $\eta$ be a measure representation of $\overline{T}$ in the sense of Proposition \ref{prop:thm_di_smirnov}. Let us now consider, for $\eta$-a.e. $\gamma$, the following auxiliary initial value problems: 
	\[
	\left\{\begin{aligned}
		\frac{\dd}{\dd t}P_t+\Lcal_{b} P_t &=0,  \qquad t \in (0,1),\\
		P_0 &=\bbb{\gamma}
	\end{aligned}\right.
	\]
	By Theorem \ref{thm:main_JFA}, we infer that $P_t =(X_t)_{*}\bbb{\gamma}$.
	By the linearity of the equation, the current
	\[
	R_t:= \int_{\Gamma} P_t \, \dd \eta(\gamma) 
	\]
	is still a solution, with initial datum $R_0 = \int_\Gamma P_0 \, \dd \eta(\gamma) = \int_\Gamma \bbb{\gamma} \, \dd \eta(\gamma) = \overline{T}$. By uniqueness, we must have $R_t=T_t$, i.e. 
	\[
	T_t = \int_{\Gamma} P_t \, \dd \eta(\gamma) = \int_{\Gamma} (X_t)_{*}\bbb{\gamma} \, \dd \eta(\gamma) = \int_{\Gamma} \bbb{\mathtt{X}_t(\gamma)} \, \dd \eta(\gamma) = \int_\Gamma \bbb{\gamma}\, \dd \eta_t(\gamma).
	\]
	Concerning the equality of masses, we want to prove that 
	\[
	\| T_t \| = \int_{\Gamma} \| \bbb{\sigma} \| \, \dd\eta_t(\sigma). 
	\] 
	It is clearly enough to show that the total masses of these two measures agree, i.e. 
	\[
	\| T_t \|(\R^d) = \int_{\Gamma} \|\bbb{\sigma} \| (\R^d) \, \dd\eta_t(\sigma).  
	\]
	Observe that, by Lemma \ref{lemma:pushforward_currents}, we have 
	\[
	\Mbf(T_t) = \Mbf((X_t)_{*} T) = \int_{\R^d}|(D_TX_t \cdot \tau) (X_{-t}(y))| \, \dd( (X_t)_{\#}\mu)(y) =  \int_{\R^d}|(D_TX_t \cdot \tau) (z)| \, \dd \mu(z).  
	\]
	From Proposition \ref{prop:thm_di_smirnov}, we have the equality 
	\[
	\mu = \|\overline{T}\| =  \int_{\Gamma} \| \bbb{\gamma} \|  \, \dd \eta(\gamma) \qquad \text{ as measures on $\R^d$,}
	\]
	and thus 
	\begin{align*}
		\Mbf(T_t) & =   \int_{\R^d}|(D_TX_t \cdot \tau) (z)| \, \dd \mu(z) \\ 
		& =  \int_{\Gamma} \dpr{ \| \bbb{\gamma} \| ,   | D_TX_t \cdot \tau (\frarg) \vert}  \, \dd  \eta(\gamma) \\ 
		& = \int_{\Gamma} \int_0^1 |(D_T X_t \cdot \tau) (\gamma(s)) | \, |\gamma'(s)| \, \dd s \, \dd  \eta(\gamma) \\ 
		& = \int_{\Gamma} \int_0^1 \bigg\vert  D_TX_t(\gamma(s)) \cdot |\gamma'(s)|  \tau(\gamma(s)) \bigg\vert \, \dd s \, \dd  \eta(\gamma) \\ 
		& = \int_{\Gamma} \int_0^1 \bigg\vert  \frac{\dd}{\dd s} X_t(\gamma(s))  \bigg\vert \, \dd s \, \dd  \eta(\gamma) \\ 
		& = \int_{\Gamma} \| \mathtt{X}_t (\gamma) \|(\R^d)  \, \dd  \eta(\gamma) \\ 
		& = \int_{\Gamma} \| \sigma \|(\R^d) \, \dd  \eta_t (\sigma)
	\end{align*}
	and this concludes the proof. 
	The passage from the third to the fourth line can be justified by means of the following differentiability result: For $\eta$-a.e. $\gamma \in \Gamma$ the map $s \mapsto X_t(\gamma(s))$ is differentiable at $\L^1$-a.e. $s \in [0,1]$ and it holds 
	\[
	\frac{\dd}{\dd s} X_t(\gamma(s)) =  D_T X_t(\gamma(s)) \cdot |\gamma'(s)| \tau(\gamma(s)). 
	\]
	By definition 
	\begin{align*}
		\frac{\dd}{\dd s} (X_t(\gamma(s))) & = \lim_h \frac{X_t\left(\gamma(s+h))-X_t(\gamma(s)\right)}{h} \\ 
		& = \lim_h \frac{X_t\left(\gamma(s) + \int_s^{s+h} |\gamma'(r)|\tau(\gamma(r)) \dd r\right)- X_t(\gamma(s)+h |\gamma'(s)|\tau(\gamma(s)))}{h}  \\
		& \qquad + \lim_h \frac{ X_t(\gamma(s) + h|\gamma'(s)|\tau(\gamma(s)))- X_t(\gamma(s))}{h} \\ 
		& =: \ell + D_TX_t(\gamma(s)) \cdot |\gamma'(s)|\tau(\gamma(s)), 
	\end{align*}
	provided that $\ell$ exists. Here we have used that the vector $h|\gamma'(s)|\tau(\gamma(s)) \in \spn(\tau(\gamma(s)))$ belongs to the decomposability bundle $V(\mu,\gamma(s))$ for $\eta$-a.e. $\gamma$ and $\L^1$-a.e. $s$. Let us show that indeed $\ell=0$. We have 
	\begin{align*}
		& \frac{1}{h} \left|X_t\left(\gamma(s) + \int_s^{s+h} |\gamma'(r)|\tau(\gamma(r))\,  \dd r\right)- X_t(\gamma(s)+h |\gamma'(s)|\tau(\gamma(s))) \right|   \\ 
		\le & \Lip(X_t) \left| \fint_s^{s+h} |\gamma'(r)|\tau(\gamma(r))\, \dd  r -  |\gamma'(s)|\tau(\gamma(s)) \right| 
	\end{align*}
	and the right-hand side goes to $0$ as $h \to 0$ at a.e. $s$ by Lebesgue's Theorem. 
\end{proof}

Combining Proposition \ref{prop:invariance} with Proposition \ref{prop:smirnov} we have that, if $\dive b=0$ and $\Lcal_b \overline{T}=0$, then $\eta_t$ are all measure representations of $\overline{T}$. In other words, pushing forward a representation of the initial current along the flow, we obtain another representation of the same current. Loosely speaking, images under the flow of $b$ of curves of $\overline{T}$ are again curves of $\overline{T}$. In the next section, we discuss a somewhat related physical illustration of this property.

\section{Alfvén Theorem as a time-dependent Frobenius' Theorem}
Consider a time-dependent vector field $B = B(t,x)$ defined on $(0,1)\times \R^3$ modelling a magnetic field in a perfectly conducting fluid moved by a (steady) velocity field $V \colon \R^3 \to \R^3$. By the Maxwell equations from Electromagnetism, the vector field $B$ is divergence-free, i.e. for every $t$, it holds $\dive B_t =0$, and according to standard modelling assumptions, $t \mapsto B_t$ solves the so-called \emph{induction equation} of Magnetohydrodynamics
\[
\frac{\dd}{\dd t} B_t =\curl (V \times B_t), \qquad t \in (0,1).  
\]
As noted in \cite{BDNR} -- modulo a sign typo -- this is precisely the coordinate expression of \eqref{eq:GTE_intro} driven by the velocity field $V$ for the three-dimensional absolutely continuous, boundaryless 1-currents $T_t := B_t \L^3$.  We have the following: 

\begin{theorem}[Alfvén Theorem]\label{thm:alfven}
	Let	$V \in \Lip(\R^3;\R^3)$ be a bounded, incompressible vector field, i.e. $\dive V=0$. Let $B \in \Crm^0([0,1]; \Lrm^{\infty}(\R^3))$ be a family of divergence-free vector fields such that the currents $T_t := B_t \L^3$ solve 
	\[
	\frac{\dd}{\dd t}T_t+\Lcal_V T_t =0,  \qquad t \in (0,1). 
	\]
	Suppose that, for every $t \in [0,1]$, the vector field $B_t$ has a unique Regular Lagrangian Flow, denoted by $Y^{(t)} := Y^{(t)}(s,y)$. Then
	\[
	Y^{(t)} = X_t \circ Y^{(0)} \qquad \text{$\L^{1+3}$-almost everywhere on $(0,1) \times \R^3$}
	\]
	for every $t \in [0,1]$.
\end{theorem}

\begin{proof}
	It is enough to show that, for fixed $t \in [0,1]$ the map $Z^{(t)} := X_t \circ Y^{(0)}$ defines a Regular Lagrangian Flow of $B_t$.
	
	Since $V$ is Lipschitz and incompressible, by Theorem \ref{thm:main_JFA} we deduce immediately the identity 
	\begin{equation}\label{eq:stellina_bis}
		(\nabla X_t \cdot B_0)(X_{-t}(\frarg))  =  B_t(\frarg) \qquad \text{ a.e. on $\R^d$.} 
	\end{equation}
	We now show that \eqref{eq:stellina_bis} implies that, for every $t \in [0,1]$ and for a.e. $y \in \R^d$, the function 
	\[
	s \mapsto X_{t}(Y^{(0)}_s(y))
	\]
	is differentiable for $\L^1$-a.e. $s$ and it holds 
	\begin{equation}\label{eq:derivata_di_composizione_cruciale_dinuovo}
		\frac{\dd}{\dd s} X_{t}(Y^{(0)}_s(y)) = B_t(X_{t}(Y^{(0)}_s(y))). 
	\end{equation}
	Indeed, we have by definition 
	\begin{align*}
		\frac{\dd}{\dd s} (X_t(Y^{(0)}_s(y))) & = \lim_h \frac{X_t(Y^{(0)}_{s+h}(y)) -X_t(Y^{(0)}_s(y))}{h} \\ 
		& = \lim_h \frac{X_t\left(Y^{(0)}_{s}(y) + \int_s^{s+h} B_0(Y^{(0)}_r(y)) \dd r\right)- X_t(Y^{(0)}_s(y)+hB_0(Y^{(0)}_s(y)))
		}{h}  \\
		& \qquad + \lim_h \frac{ X_t(Y^{(0)}_s(y)+hB_0(Y^{(0)}_s(y)))- X_t(Y^{(0)}_s(y))}{h} \\ 
		& =: \ell + (\nabla X_t \cdot B_0)(Y^{(0)}_s(y)), 
	\end{align*}
	provided that $\ell$ exists. Let us show that $\ell=0$. We have 
	\begin{align*}
		& \frac{1}{h}\left|X_t\left(Y^{(0)}_{s}(y) + \int_s^{s+h} B_0(Y^{(0)}_r(y)) \dd r \right) - X_t(Y^{(0)}_s(y)+hB_0(Y^{(0)}_s(y)))\right| \\ 
		\le & \Lip(X_t) \left| \fint_s^{s+h} B_0(Y^{(0)}_r(y))\, \dd  r - B_0(Y^{(0)}_s(y))\right| 
	\end{align*}
	and the right-hand side goes to $0$ as $h \to 0$ at every Lebesgue point of $r \mapsto B_0(Y^{(0)}_r(y))$, and thus for a.e. $s$. 
	We have therefore obtained 
	\[
	\frac{\dd}{\dd s} (X_t(Y^{(0)}_s(y)) = (\nabla X_t \cdot B_0)(Y^{(0)}_s(y)) = B_t(X_t(Y^{(0)}_s(y)))
	\]
	for a.e. $y \in \R^d$ and for a.e. $s \in [0,1]$, where we have used again a Fubini-like argument like the one in the proof of Theorem \ref{thm:big_one}.
	The fact that the maps $X_t(Y^{(0)}(s,\frarg))$ are measure preserving on $\R^d$ is immediate and therefore the conclusion follows.
\end{proof}

In other words, chosen a (magnetic) field line at time $0$, its image under $X_t$ is a (magnetic) field line at time $t$. This could be interpreted by saying that the magnetic field lines evolve under the fluid flow, i.e., they are
‘frozen in the flow.’ This is in accordance with the so-called Alfvén Theorem from MHD, which asserts the following:
\vspace{1em}
\begin{quote}
	Suppose that we have a homogeneous magnetic field in a perfectly conducting liquid. The magnetic lines of force can be considered as elastic strings according to the usual mechanical picture of electrodynamical phenomena. In view of the infinite conductivity, every motion (perpendicular to the field) of the liquid in relation to the lines of force is forbidden because it would give infinite eddy currents. Thus the matter of the liquid is \emph{fastened} to the lines of force, constituting a series of strings.

	\attrib{H. Alfvén,
		{\bfseries Nature} 150, 405–406 (1942),  \cite{alfven}}
\end{quote}
\vspace{1em}
Notice, however, that this is no longer true as such in the compressible case. From the induction equation one deduces this time that for every $t \in [0,1]$ it holds 
\[
B_t =( \nabla X_t \cdot B_0 )(X_{-t}) \rho_t 
\]
almost everywhere on $\R^d$. Assuming that $V$ is smooth, one can divide by the density $\rho_t$ and conclude, in this case, that the frozen-in lines are the integral curves of the rescaled vector field 
\[
\tilde{B}_t := \frac{1}{\rho_t} B_t.
\] 
A similar conclusion can be reached in the Lipschitz setting, provided the rescaled vector field $\tilde{B}_t$ has a unique Regular Lagrangian Flow. In other words, only the \emph{direction} of the field lines is frozen, but not their intensity.

\bibliographystyle{plain}

\end{document}